\documentclass[11pt,a4paper]{article}

\usepackage{geometry} 
\usepackage{amsmath,amsfonts,amsthm,amssymb,yhmath, mathtools} 
\usepackage[dvipsnames]{xcolor}  
\usepackage{enumitem} 
\usepackage{graphicx, float, subfig, overpic} 
\usepackage{stmaryrd} 
\usepackage{faktor} 
\usepackage{mathrsfs} 
\usepackage{titlesec} 
\usepackage[font=small,labelfont=bf,tableposition=top]{caption}
\usepackage{authblk} 
\usepackage{todonotes,color} 
\usepackage{hyperref, comment} 
\hypersetup{
	colorlinks=true,
	linkcolor=blue,
	citecolor=green,
}

\usepackage[numbers]{natbib}
\numberwithin{equation}{section}
\newcommand{\st}{\h:\h}

\newcommand{\h}{\hspace{1mm}}
\newcommand{\hh}{\hspace{5mm}}

\newcommand{\qqand}{\qquad \text{and} \qquad}

\theoremstyle{plain}
\newtheorem{lemma}{Lemma}[section]
\newtheorem{corollary}[lemma]{Corollary}
\newtheorem{proposition}[lemma]{Proposition}
\newtheorem{remark}[lemma]{Remark}

\newtheorem{theorem}[lemma]{Theorem}

\newtheoremstyle{named}{}{}{\itshape}{}{\bfseries}{.}{.5em}{\thmname{#1}\thmnumber{ #2}. (\thmnote{#3})}
\theoremstyle{named}

\newtheorem*{namedTheorem*}{Theorem}

\def\be{\begin{eqnarray}}
\def\ee{\end{eqnarray}}
\def\beal{\begin{aligned}}
\def\enal{\end{aligned}}

\newcommand{\norm}[1]{\left\lVert#1\right\rVert}

\newcommand{\vabs}[1]{\left| #1 \right|}
\newcommand{\vabss}[1]{| #1 |}
\newcommand{\paren}[1]{\left(#1\right)}
\newcommand{\claus}[1]{\left\{#1\right\}}
\newcommand{\boxClaus}[1]{\left[#1\right]}

\newcommand{\tl}{\tilde}
\newcommand{\wt}{\widetilde}

\renewcommand{\Re}{\mathrm{Re\, }}
\renewcommand{\Im}{\mathrm{Im\,}}
\renewcommand{\arg}{\mathrm{arg\,}}

\newcommand{\reals}{\mathbb{R}}

\newcommand{\complexs}{\mathbb{C}}
\newcommand{\torus}{\mathbb{T}}

\newcommand{\al}{\alpha}

\newcommand{\tht}{\theta}

\newcommand{\phiA}{\varphi}

\newcommand{\AAA}{\mathcal{A}}
\newcommand{\BB}{\mathcal{B}}

\newcommand{\FF}{\mathcal{F}}
\newcommand{\GG}{\mathcal{G}}

\newcommand{\HH}{\mathcal{H}}
\newcommand{\II}{\mathcal{I}}
\newcommand{\JJ}{\mathcal{J}}
\newcommand{\KK}{\mathcal{K}}
\newcommand{\LL}{\mathcal{L}}

\newcommand{\OO}{\mathcal{O}}

\newcommand{\RRR}{\mathcal{R}}

\definecolor{myGreen}{RGB}{0, 200, 0}
\definecolor{myOrange}{RGB}{255, 100, 0}
\definecolor{myYellow}{RGB}{255, 200, 0}
\definecolor{myBlue}{RGB}{0, 200, 255}
\definecolor{myPurple}{RGB}{200, 0, 200}

\usepackage{stackengine,scalerel}
\stackMath

\newcommand{\kInn}{\kappa}
\newcommand{\kInnS}{\kappa^*}
\newcommand{\CInn}{\Theta}
\newcommand{\CInnHat}{\widetilde{\Theta}}
\newcommand{\inn}{\mathrm{in}}
\newcommand{\unstable}{{\mathrm{u}}}
\newcommand{\stable}{{\mathrm{s}}}

\newcommand{\DuInn}{\mathcal{D}^{\mathrm{u}}_{\kInn}}
\newcommand{\DuInnInf}{\mathcal{D}^{\mathrm{u}}_{\kInn,\eta}}

\newcommand{\DsInn}{\mathcal{D}^{\mathrm{s}}_{\kInn}}
\newcommand{\DsInnInf}{\mathcal{D}^{\mathrm{s}}_{\kInn,\eta}}
\newcommand{\DdInn}{\mathcal{D}^{\diamond}_{\kInn}}
\newcommand{\EInn}{\mathcal{E}_{\kInn}}

\newcommand{\rectangle}{\mathbf{R}_{\kappa}(\varrho_1,\varrho_2)}

\newcommand{\Zu}{Z^{\mathrm{u}}}
\newcommand{\ZuHat}{\widetilde{Z}^{\mathrm{u}}}
\newcommand{\Zs}{Z^{\mathrm{s}}}

\newcommand{\Zd}{Z^{\diamond}}
\newcommand{\Wu}{W^{\mathrm{u}}}
\newcommand{\WuHat}{\widetilde{W}^{\mathrm{u}}}

\newcommand{\Wd}{W^{\diamond}}
\newcommand{\Xu}{X^{\mathrm{u}}}
\newcommand{\XuHat}{\widetilde{X}^{\mathrm{u}}}

\newcommand{\Xd}{X^{\diamond}}

\newcommand{\Yu}{Y^{\mathrm{u}}}
\newcommand{\YuHat}{\widetilde{Y}^{\mathrm{u}}}

\newcommand{\Yd}{Y^{\diamond}}
\newcommand{\DZ}{\Delta Z}
\newcommand{\DZo}{\Delta Z_{0}}
\newcommand{\DW}{\Delta W}

\newcommand{\DX}{\Delta X}

\newcommand{\DY}{\Delta Y}
\newcommand{\DYo}{\Delta Y_{0}}

\newcommand{\XcalU}{\mathcal{X}^{\mathrm{u}}}
\newcommand{\Ycal}{\mathcal{Y}}
\newcommand{\Zcal}{\mathcal{Z}}
\newcommand{\XcalUTotal}{\mathcal{X}^{\mathrm{u}}_{\times}}

\newcommand{\normInn}[1]{\lVert#1\rVert}

\newcommand{\normInnU}[1]{\lVert#1\rVert^{\mathrm{u}}}

\newcommand{\normInnTotal}[1]{\lVert#1\rVert_{\times}}
\newcommand{\normInnTotalSmall}[1]{\lVert#1\rVert_{\times}}

\newcommand{\normInnDiff}[1]{\lVert#1\rVert}
\newcommand{\normInnDiffSmall}[1]{\lVert#1\lVert}
\newcommand{\normInnDiffExp}[1]{\llbracket#1\rrbracket}

\newcommand{\cttInnExistA}{b_1}
\newcommand{\cttInnExistB}{b_2}
\newcommand{\cttInnDiff}{b_3}

\title{Breakdown of homoclinic orbits to $L_3$: Nonvanishing of the Stokes constant} 
\date{\today}

\author[1,5]{Inmaculada Baldom\'a}
\author[2]{Maciej J. Capi\'nski}
\author[3]{Mar Giralt\thanks{Corresponding author.\\
		\textbf{E-mail adresses:} \href{mailto:immaculada.baldoma@upc.edu}{immaculada.baldoma@upc.edu} (I. Baldom\'a), 
		\href{mailto:mcapinsk@agh.edu.pl}{mcapinsk@agh.edu.pl} (M. Capi\'nski)
		\href{mailto:mar.giralt@obspm.fr}{mar.giralt@obspm.fr} (M. Giralt),
		\href{mailto:guardia@ub.edu}{guardia@ub.edu} (M. Guardia). 
		\textbf{MSC classes:} 37J46, 37N05.}}
\author[4,5]{Marcel Guardia}
\affil[1]{Departament de Matem\`atiques \& IMTECH, Universitat Polit\`ecnica de Catalunya, Diagonal 647, 08028 Barcelona, Spain}
\affil[2]{Faculty of Applied Mathematics, AGH University of Krak\'ow, al. Mickiewicza 30, 30-059 Krak\'ow, Poland}
\affil[3]{IMCCE, CNRS, Observatoire de Paris, Universit\'e PSL, Sorbonne Universit\'e, 77 Avenue Denfert-Rochereau, 75014 Paris, France}
\affil[4]{Departament de Matem\`atiques i Inform\`atica, Universitat de Barcelona, Gran Via, 585, 08007 Barcelona, Spain}
\affil[5]{Centre de Recerca Matem\`atica, Campus de Bellaterra, Edifici C, 08193 Barcelona, Spain}

\begin{document}

\maketitle

\begin{abstract}
The Restricted Planar Circular 3-Body Problem models the motion of a body of negligible mass
under the gravitational influence of two massive bodies, called the primaries, which perform circular orbits coplanar with that of the massless body.  In
rotating coordinates, it can be modelled by a two degrees of freedom Hamiltonian system,
which has five critical points called the Lagrange points. Among them, the point $L_3$ is a saddle-center  which is collinear with the
primaries and beyond the largest of the two. The papers \cite{articleInner,articleOuter} provide an asymptotic formula for the distance between the one dimensional stable and unstable manifolds of $L_3$ in a transverse section for small
values of the mass ratio $0 < \mu\ll 1$. This distance is exponentially small with respect to $\mu$ and its first order depends on what is usually called a Stokes constant.  The non-vanishing of this constant implies that the distance between the invariant manifolds at the section is not zero.
In this paper, we prove that the Stokes constant is non-zero. The proof is computer assisted.
\end{abstract}

\tableofcontents

\section{Introduction and main result}\label{sec:intro}

The Restricted Circular $3$-Body Problem 
models the motion of a body of negligible mass under the gravitational influence of two massive bodies, called the primaries, which perform a circular motion.
If one also assumes that the massless body moves on the same plane as the primaries one has the Restricted Planar Circular $3$-Body Problem (RPC$3$BP).

Let us name the two primaries $S$ (star) and $P$ (planet) and normalize their masses so that $m_S=1-\mu$ and $m_P=\mu$, with $\mu \in \left( 0, \frac{1}{2} \right]$. 
In a rotating coordinate system, the positions of the primaries can be fixed at $q_S=(\mu,0)$ and  $q_P=(\mu-1, 0)$. 
Then, the position and momenta of the third body, $(q,p) \in \reals^2 \times \reals^2$, are governed by the Hamiltonian system associated to the autonomous Hamiltonian
\begin{equation}\label{def:hamiltonianInitialNotSplit} 	\begin{split}
		H(q,p;\mu) &= \frac{||p||^2}{2} 
		- q^t \left( \begin{matrix} 0 & 1 \\ -1 & 0 \end{matrix} \right) p 
		-\frac{(1-\mu)}{||q-(\mu,0)||} 
		- \frac{\mu}{||q-(\mu-1,0)||}.
	\end{split}
\end{equation}
For $\mu>0$, it is a well known fact that \eqref{def:hamiltonianInitialNotSplit}  has five critical points, usually called Lagrange points
(see Figure~\ref{fig:L3Outer}). 
The three collinear Lagrange points, $L_1$, $L_2$ and $L_3$, are of center-saddle type whereas, for small $\mu$, the triangular ones, $L_4$ and $L_5$, are of center-center type  (see, for instance, \cite{Szebehely}).

\begin{figure}[H]
\centering
\begin{overpic}[scale=0.5]{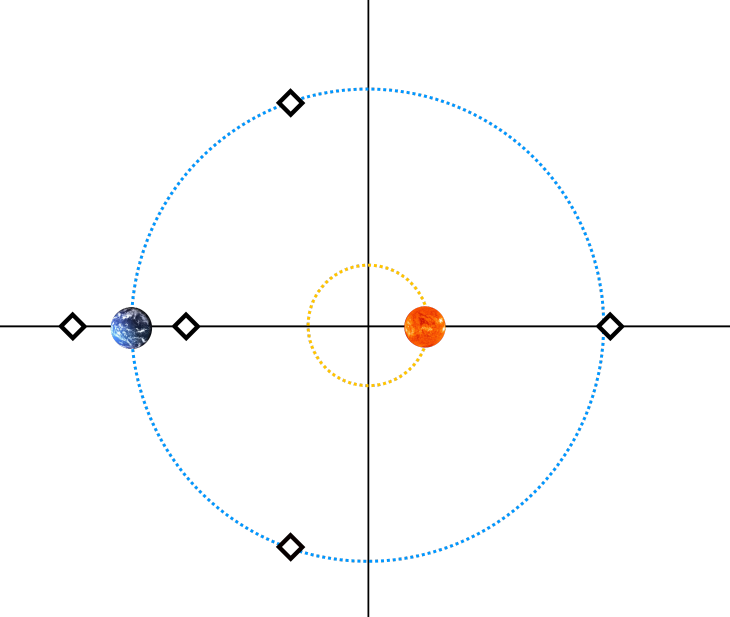}
	\put(60,33){{\color{orange} $S$ }}
	\put(11,33){{\color{blue} $P$ }}
	\put(23,45){{\color{black} $L_1$ }}
	\put(7,45){{\color{black} $L_2$ }}
	\put(87,42){{\color{black} $L_3$ }}
	\put(37,75){{\color{black} $L_5$ }}	
	\put(37,2){{\color{black} $L_4$ }}
\end{overpic}
\caption{Projection onto the $q$-plane of the Lagrange equilibrium points for the RPC$3$BP on rotating coordinates.}
	\label{fig:L3Outer}
\end{figure}
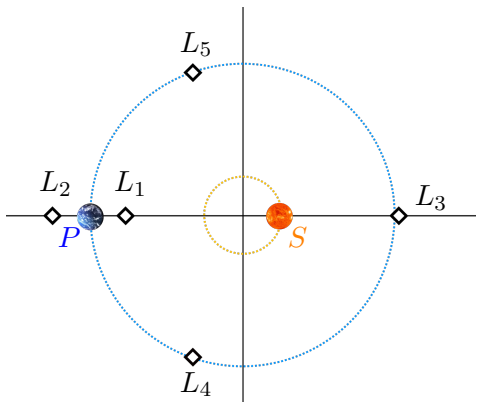

The invariant manifolds of the (unstable) Lagrange points are of fundamental im\-por\-tance for understanding the dynamics of the RPC3BP. In particular, those of the point $L_3$ (more precisely its center-stable and center-unstable invariant manifolds) act as boundaries of \emph{effective stability} of the stability domains around $L_4$ and $L_5$ 
(see \cite{GJMS01v4, SSST13}). They also allow to create transfer orbits from the small primary to $L_3$ in the RPC$3$BP (see \cite{HTL07, TFRPGM10})  or between primaries in the Bicircular 4-Body Problem (see \cite{JorNic20, JorNic21}).

In understanding how the invariant manifolds of $L_3$ structure the global dynamics, it is fundamental to know whether they coincide or not. The purpose of the papers \cite{articleInner,articleOuter}  and the present one is to prove that these invariant manifolds do not coincide the first time they hit a suitable transverse section. This is the content of the Theorems \ref{TheoremA} and \ref{thm:Stokes} below. This fact has several dynamical implications, which are proven in the paper  \cite{articleChaos} and are explained in Remarks \ref{rmk:implications} and \ref{rmk:chaos} below.

The manifolds $W^{\unstable}(L_3)$ and $W^{\stable}(L_3)$ lie in the so called $1:1$ mean motion resonance and have two branches each. One pair, which we denote by $W^{\unstable,+}(L_3)$ and $W^{\stable,+}(L_3)$, circumvents $L_5$ whereas the other circumvents $L_4$ and it is denoted as  $W^{\unstable,-}(L_3)$ and $W^{\stable,-}(L_3)$, see Figure~\ref{fig:perturbedInvariantManifolds1d}.
These branches are symmetric with respect to 
\begin{equation*}
	\Psi(q,p)=(q_1,-q_2,-p_1,p_2).
\end{equation*}
%
Thus, to compute the distance between the manifolds, one can restrict the study to the first ones, $W^{\unstable,+}(L_3)$ and $W^{\stable,+}(L_3)$.
We measure this distance in symplectic polar  coordinates, defined as
\begin{align*}
	q= 
	r \begin{pmatrix}
		\cos \tht \\ 
		\sin \tht
	\end{pmatrix},
	\qquad
	p = 
	R
	\begin{pmatrix}
		\cos \tht \\ 
		\sin \tht
	\end{pmatrix} 
	- \frac{G}{r} \begin{pmatrix}
		\sin \tht \\ 
		-\cos \tht
	\end{pmatrix},
\end{align*} 
where 
$r$ is the distance of the third body from the origin, $\tht$ its argument, 
$R$ is the radial linear momentum  and $G$ is the angular momentum.

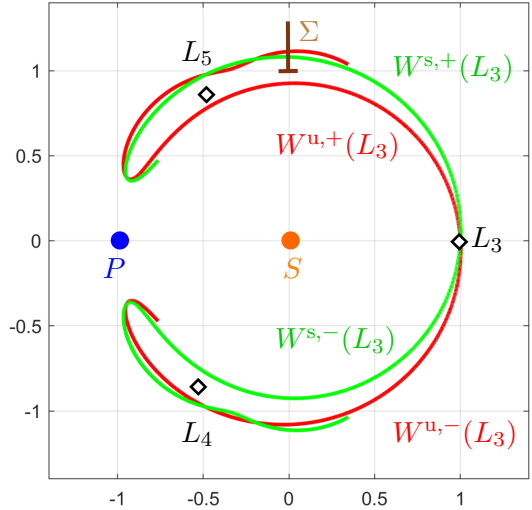
\begin{figure}[H]
	\centering
	\begin{overpic}[scale=0.3]{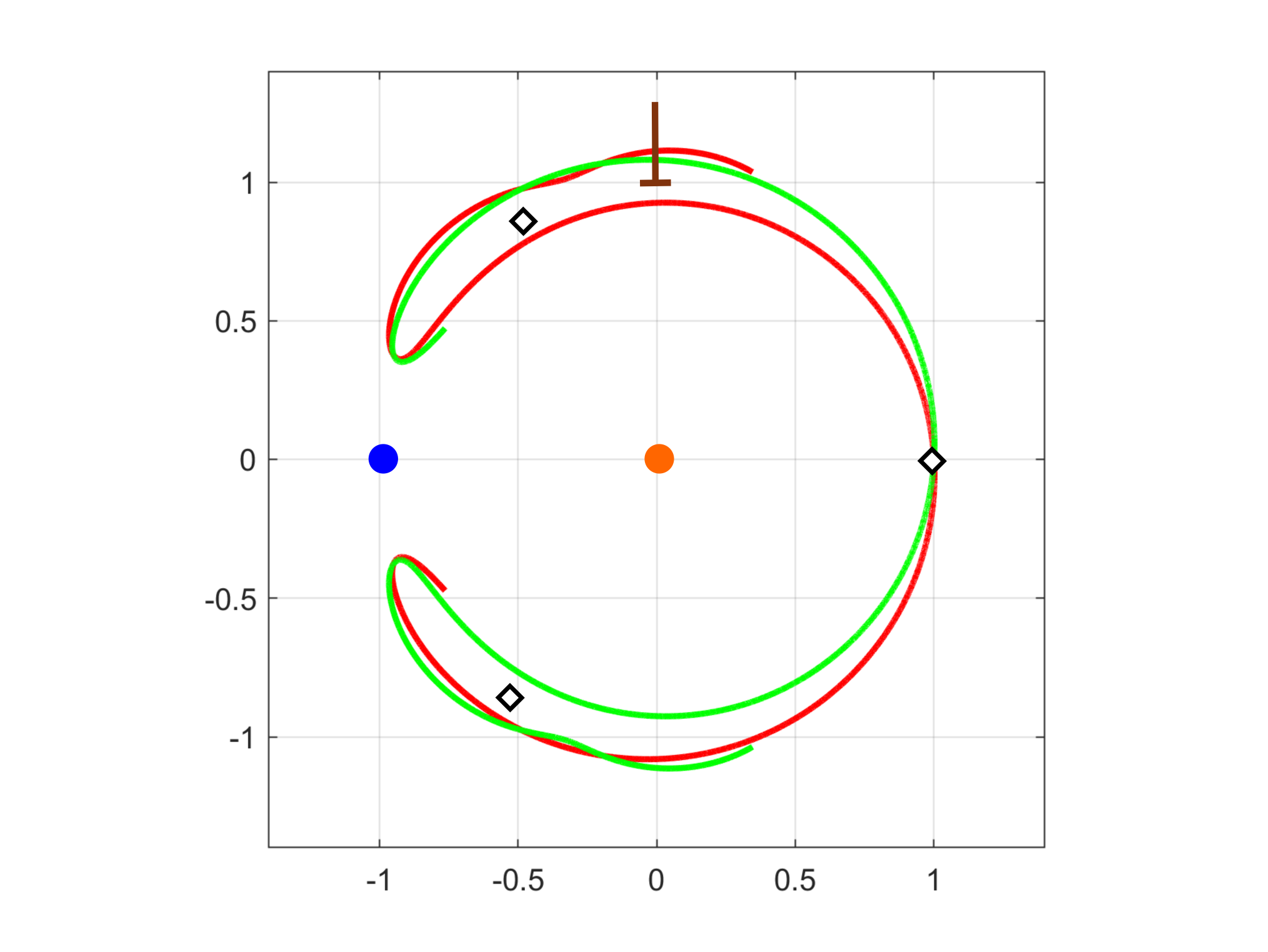}
		\put(51,34){\color{orange} $S$ }
		\put(28,34){\color{blue} $P$ }
		\put(75,38){{\color{black} $L_3$ }}
		\put(38,62){{\color{black} $L_5$ }}	
		\put(38,13){{\color{black} $L_4$ }}
		\put(53,64.5){{\color{brown} $\Sigma$ }}
		\put(65,60){\color{myGreen} $W^{\stable,+}(L_3)$}
		\put(50,50){\color{red} $W^{\unstable,+}(L_3)$}
		\put(50,25){\color{myGreen} $W^{\stable,-}(L_3)$}
		\put(65,13){\color{red} $W^{\unstable,-}(L_3)$}
	\end{overpic}
	\caption{Projection onto the $q$-plane of the unstable (red) and stable (green) manifolds of $L_3$, for ${\mu=0.0028}$. }
	\label{fig:perturbedInvariantManifolds1d}
\end{figure}

We consider as well the $3$-dimensional section 

\begin{equation*}
\Sigma = \claus{(r,\tht,R,G) \in \reals \times \torus \times \reals^2 
	\st r>1, \, \tht=\frac{\pi}2 \,}
\end{equation*}

and denote by $(r^{\unstable}_*,\frac{\pi}2, R^{\unstable}_*,G^{\unstable}_*)$ and $(r^{\stable}_*,\frac{\pi}2,R^{\stable}_*,G^{\stable}_*)$ the first crossing of the invariant manifolds with this section (see Figure \ref{fig:perturbedInvariantManifolds1d}). 
The next theorem, proven in \cite{articleInner, articleOuter}, measures the distance between these points for $0< \mu\ll 1$.

\begin{theorem}	\label{TheoremA}
	There exists $\mu_0>0$ such that, for $\mu \in (0,\mu_0)$,

	\[	\norm{(r^{\unstable}_*,R^{\unstable}_*,G^{\unstable}_*)-(r^{\stable}_*,R^{\stable}_*,G^{\stable}_*)}
	=
	\sqrt[3]{4} \,
	\mu^{\frac13} e^{-\frac{A}{\sqrt{\mu}}} 
	\boxClaus{\vabs{\CInn}+\OO\paren{\frac1{\vabs{\log \mu}}}},
	\]
	where
	\begin{itemize}
		\item The constant $A>0$ is given by the real-valued integral
		\begin{equation}\label{def:integralA}
			A= \int_0^{\frac{\sqrt{2}-1}{2}} \frac{2}{1-x}\sqrt\frac{x}{3(x+1)(1-4x-4x^2)}  dx\approx 0.177744.
		\end{equation}
		\item The constant $\CInn \in \complexs$ is the Stokes constant associated to the inner equation analyzed in \cite[Theorem 2.7]{articleInner} (see also Theorem~\ref{theorem:mainAnalytic}). 
	\end{itemize}
\end{theorem}

Theorem \ref{TheoremA} provides a first order for the distance between the invariant manifolds of $L_3$, at the first crossing with $\Sigma$, provided the Stokes constant $\Theta$ is not zero. The main result of the present paper is the following.  

\begin{theorem}\label{thm:Stokes}
The constant $\Theta\in\mathbb{C}$ introduced in Theorem \ref{TheoremA}     satisfies
\[
\Theta\neq 0.
\]
\end{theorem}

The proof of this theorem relies on recent techniques developed by some of the authors in \cite{BCGS23}. Note however, that the models considered in that paper are rather simple unfoldings of the Hopf-zero singularity whereas here we deal with a Celestial Mechanics model, which makes the analysis considerably more delicate.


\begin{remark}\label{rmk:implications}
Theorem~\ref{TheoremA} and Theorem~\ref{thm:Stokes} imply that the unstable and unstable manifolds of $L_3$ do not meet at their first encounter with the section $\Sigma$.
This is expected as, in generic systems, homoclinic connections are likely to breakdown when subjected to small perturbations. 
%
%
However, this does not prevent the existence of homoclinic connections.
Indeed, in \cite{articleChaos}, by relying on the results in the present paper and \cite{articleInner, articleOuter}, some of the authors of this paper  construct homoclinic orbits to $L_3$ for a sequence of mass ratios $\mu$ tending to zero. These homoclinic orbits are two round, that is cross the section $\Sigma$ twice.
\end{remark}
\begin{remark}\label{rmk:chaos}
A fundamental question in dynamical systems is to prove that a system exhibits chaotic dynamics. 
One of the  classical ways to prove this  it is to show the existence of transverse intersections between the stable and unstable manifolds of a particular invariant object (see \cite{GuckenheimerHolmes}). 
In the case of $L_3$, its stable and unstable manifolds are $1$-dimensional and, therefore, they have too small dimension to allow for transverse intersections. 
However, in \cite{articleChaos}, we are able to prove that the splitting of the invariant manifolds of $L_3$ imply that there exist transverse intersections between the stable and unstable manifolds of certain periodic Lyapunov orbits close  to $L_3$.  
This proves the existence of chaotic dynamics close to $L_3$ and its invariant manifolds. Moreover, we are also able to prove that some of these Lyapunov periodic orbits possess quadratic homoclinic tangencies, which imply the existence of Newhouse domains in the RPC3BP.
\end{remark}

The paper is organized as follows. In Section~\ref{sec:inner} we summarize the main steps to prove Theorem~\ref{TheoremA}, performed in the previous works~\cite{articleInner, articleOuter}; we present the inner equation, which is independent on the small parameter $\mu$, and, finally, we describe the relation between the Stokes constant $\Theta$ and suitable solutions, $Z^{\mathrm{u},\mathrm{s}}$ of the inner equation(see Theorem~\ref{theorem:mainAnalytic}) below. The (short) Section~\ref{sec:proofthmStokes} is devoted to explain the strategy to prove Theorem~\ref{thm:Stokes} which consists in two main steps: 
characterization of the complex domains where $Z^{\mathrm{u},\mathrm{s}}$ are defined (Theorem~\ref{thm:FixedPointQuantitative}) and analysis of its difference (Theorem~\ref{thm:difference}).  
Then, in Section~\ref{sec:proof}, relying on the approach developed in~\cite{BCGS23}, we prove Theorem~\ref{thm:FixedPointQuantitative}.
In Section~\ref{sec:difference} we give the proof of Theorem~\ref{thm:difference}, which in part is computer assisted\footnote{The code for the computer assisted part of the proof is available on the personal web page of MJC.}.



\section{The invariant manifolds of $L_3$ and the inner equation}\label{sec:inner}

Theorem \ref{TheoremA} falls into what is usually called exponentially small splitting of separatrices. That is, on the perturbative analysis of the distance between the stable and unstable manifold of an invariant object when it is exponentially small.
Its proof follows the original approach proposed by Lazutkin in his seminal work on the Standard Map \cite{Laz84, Laz05}. Note that the papers \cite{articleInner, articleOuter} are the first ones where such approach has been implemented in a Celestial Mechanics model. We refer to \cite{articleInner} for a detailed list of references on the exponentially small splitting of separatrices phenomenon.

Let us summarize the main steps of this proof and explain where the Stokes constant arises. 

Note first that, in  the limit problem $H$ in \eqref{def:hamiltonianInitialNotSplit} with $\mu=0$, the five Lagrange points and the associated invariant manifolds  ``collapse'' into the circle of (degenerate) critical points  $\norm{q}=1$ and $p=(p_1,p_2)=(-q_2,q_1)$.
Therefore, to analyze the invariant manifolds, it is convenient  to perform a singular change of coordinates to obtain a ``new first order'' Hamiltonian which  has a  saddle-center equilibrium point with stable and unstable manifolds that coincide along a separatrix.
This change of coordinates boils down to a suitable (singular with respect to $\mu$) scaling of the classical  Poincar\'e planar elements (see~\cite{MeyerHallOffin}). These coordinates are explained in full detail in \cite{articleInner}.

In these coordinates, the proof of Theorem \ref{TheoremA} relies on the following steps. 
\begin{enumerate}[label*=\Alph*.]
	\item We perform the aformentioned change of coordinates which captures the slow-fast dynamics of the system.
	The new Hamiltonian becomes a (fast) oscillator weakly coupled to a $1$-degree of freedom Hamiltonian with a saddle point and a separatrix associated to it.
	\item We analyze the analytical continuation of a time-parametrization of the separatrix.
	In particular, we obtain its maximal strip of analyticity, which is given by $|\Im t|<A$ where $A$ is the constant introduced in \eqref{def:integralA} and $t$ the time of the parametrization.
	We also describe the character and location of the complex singularities  at the boundaries of this region.

	\item We derive the inner equation, which gives the first order of the original system close to the singularities of the separatrix described in Step B.
	This equation is independent of the perturbative parameter $\mu$.
	
	\item We study two special solutions of the inner equation which are approximations of the perturbed invariant manifolds near the singularities.
	Moreover, we provide an asymptotic formula for the difference between these two solutions of the inner equation.
	This difference is given in terms of the Stokes constant $\Theta$ introduced in Theorem \ref{TheoremA}.

\item We prove the existence of the analytic continuation of suitable parametrizations of $W^{\unstable,+}(L_3)$ and $W^{\stable,+}(L_3)$ in appropriate complex domains (and as graphs).
These domains contain a segment of
the real line and intersect a neighborhood sufficiently close to the singularities of the  separatrix.
\item By using complex matching techniques, we compare the solutions of the inner equation with the graph parametrizations of the perturbed invariant manifolds.
\item Finally, we prove that the dominant term of the difference between manifolds is given by the term obtained from the difference of the solutions of the inner equation.
\end{enumerate}

Steps A, B, C and D are performed in \cite{articleInner} whereas the Steps E, F and G are performed in \cite{articleOuter}. In particular, \cite{articleInner} showed that the constant $\Theta\in\mathbb{C}$ exists but no proof of its non-vanishing is provided.

To prove that the Stokes constant $\Theta$ is not zero, we have to perform a deeper analysis of the two special solutions of the inner equation mentioned in Step D.  To this end, we first introduce the so-called inner Hamiltonian, the computation of which (Step C) is explained in full detail in \cite{articleInner}, given by
\begin{equation}\label{def:hamiltonianInner}
	\HH(U,W,X,Y) = W + XY + \KK(U,W,X,Y),
\end{equation}
with
\begin{equation}
	\KK(U,W,X,Y) = 
	-\frac{3}{4}U^{\frac{2}{3}} W^2 
	- \frac{1}{3 U^{\frac{2}{3}}}
	\paren{\frac{1}{\sqrt{1+\JJ(U,W,X,Y)}} - 1 } \label{def:hamiltonianK}
\end{equation}
and
\begin{equation}\label{def:hFunction}
	\begin{split}
		\JJ(U,W,X,Y) =& \,  
		\frac{4 W^2}{9 U^{\frac{2}{3}} } 
		-\frac{16 W}{27 U^{\frac{4}{3}}}  
		+\frac{16}{81 U^{2}}
		%
		+\frac{4(X+Y)}{9 U}
		\paren{W -\frac{2}{3 U^{\frac{2}{3}}}} \\[0.5em]
		&- \frac{4i(X-Y)}{3 U^{\frac{2}{3}}}
		-\frac{X^2+Y^2}{3 U^{\frac{4}{3}}}
		+\frac{10 XY}{9 U^{\frac{4}{3}}},
	\end{split}
\end{equation}
and the symplectic form 
\[
\Omega=dU\wedge dW+i dX\wedge dY.
\]

Let us now summarize the main features of the analysis of the inner equation performed previously in~\cite{articleInner} (see Theorem~\ref{theorem:mainAnalytic} below). 
Following the approach presented in~\cite{BalSea08} (see also~\cite{Bal06}), we look for two suitable solutions of the Hamiltonian system associated to $\mathcal{H}$, analyzing their orbits as graphs over $U$. That is, we do not analyze the trajectories of $\HH$ directly. 
To this end, we introduce $Z=(W,X,Y)$ and the matrix
 
\begin{equation*}
	\AAA= \begin{pmatrix}
		0 & 0 & 0 \\
		0 & i & 0 \\
		0 & 0 & -i
	\end{pmatrix}.
\end{equation*}
 
With this notation, the equation associated to the Hamiltonian $\HH$ can be written as
 
\begin{equation}\label{eq:systemEDOsInner}
	\left\{ \begin{array}{l}
		\dot{U} = 1 + g(U,Z),\\
		\dot{Z} = \AAA Z + f(U,Z),
	\end{array} \right.
\end{equation}
 
where 
$f = \paren{-\partial_U \KK, 
	i \partial_Y \KK, -i\partial_X \KK }^T$ 
and
$g = \partial_{W} \KK$.
%
%
%
Therefore, to look for solutions of this equation parametrized as graphs with respect to $U$, we search functions
 
\begin{equation*}
	\Zd(U) = \big(\Wd(U),\Xd(U),\Yd(U)\big)^T,
	\qquad
	\text{for } \diamond=\unstable,\stable,
\end{equation*}
 
satisfying the invariance condition given by~\eqref{eq:systemEDOsInner}, that is
\begin{equation}\label{eq:invariantEquationInner}
	\partial_U \Zd = 
	\AAA \Zd + \RRR[\Zd],
	\qquad
	\text{for } \diamond=\unstable,\stable,
\end{equation}
where
\begin{equation}\label{def:operatorRRRInner}
	\RRR[\varphi](U)= 
	\frac{f(U,\varphi)- g(U,\varphi) \AAA \varphi }{1+g(U,\varphi)}.
\end{equation}

The special solutions $\Zd$ we are interested in, satisfy the asymptotic conditions
\begin{equation}\label{eq:asymptoticConditionsInner}
	\begin{split}
		\lim_{\Re U \to -\infty} \Zu(U) = 0, 
		\qquad 
		\lim_{\Re U \to +\infty} \Zs(U) = 0. 
	\end{split}
\end{equation}
%
%
In fact, for a fixed $\gamma \in \paren{0,\frac{\pi}{2}}$, we look for functions $\Zu$ and $\Zs$ satisfying~\eqref{eq:invariantEquationInner},
\eqref{eq:asymptoticConditionsInner} and defined in the domains
\begin{equation}\label{def:domainInnner}
	\begin{split}
		&\DuInn = \claus{ U \in \complexs \text{ : }
			|\Im U| \geq \tan \gamma \,\Re U + \frac{\kInn}{\cos \gamma}, \,
            \Re U \leq 0},
            \quad 
		\DsInn = -\DuInn,
	\end{split}
\end{equation}
respectively, for some $\kInn>0$ big enough (see Figure~\ref{fig:dominiInnerUnstable}). 

\begin{figure}[t] 
	\centering
	\begin{overpic}[scale=0.8]{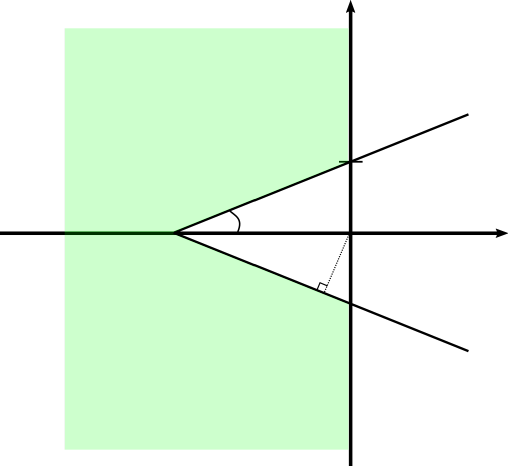}
		\put(20,60){$\DuInn$}
		\put(50,48){\small {$\gamma$}}
		\put(73,58){\small{$\rho(\kappa,\gamma) = \frac{\kInn}{\cos\gamma}$}}
        \put(62,40){\small{$\kInn$}}
		\put(102,45){\small$\Re U$}
		\put(67,93){\small $\Im U$}
	\end{overpic}
	\bigskip
	\caption{The inner domain, $\DuInn$, for the unstable case (see \eqref{def:domainInnner}).}
	\label{fig:dominiInnerUnstable}
\end{figure} 
As a consequence,  the difference $\DZ=\Zu-\Zs$ can be analyzed in the overlapping domain
 
\begin{equation*}
	\EInn = \claus{ U \in \complexs \st  \Im U \leq -\rho(\kappa,\gamma), \, \Re U =0},
\end{equation*}
 
where
\begin{equation}\label{def:rho}
\rho(\kappa,\gamma)=\frac{\kappa}{\cos\gamma}.
\end{equation}

\begin{remark}\label{rmk_before_ThInner}
    The domains used in~\cite{articleInner} are bigger than the ones consider in the present work and they have an open overlapping domain. We have chosen these smaller domains, $\DuInn$ and $\DsInn$. The resason is that, on the one hand, we will see that it is enough to analyze the difference $\Zu - \Zs $ on $\EInn=\DuInn\cap \DsInn$ (contained in the imaginary axis) and, on the other, to restrict the analysis to smaller domains, makes the explicit computation of all the constants that appear in our analysis easier.  
\end{remark}



The following result is proven in~\cite{articleInner}. 
\begin{theorem}\label{theorem:mainAnalytic}
	There exist $\kInn_0, \cttInnExistA, \cttInnExistB>0$
	such that for any $\kInn\geq\kInn_0$,	
	equation~\eqref{eq:invariantEquationInner} has analytic solutions
	$
	\Zd(U) =(\Wd(U),\Xd(U),\Yd(U))^T, 
	$ for $U \in \DdInn$, $\diamond=\unstable,\stable$, satisfying
  
	\begin{equation*}
		| U^{\frac{8}{3}} \Wd(U)| \leq \cttInnExistA, \qquad
		| U^{\frac{4}{3}} \Xd(U) | \leq \cttInnExistB, \qquad
		| U^{\frac{4}{3}} \Yd(U) | \leq \cttInnExistB.
	\end{equation*}
  
	In addition, there exist $\CInn \in \complexs$ and $\cttInnDiff>0$ independent of $\kInn$, and a function $\chi=(\chi_1,\chi_2,\chi_3)^T$ 
	such that
	\begin{equation}\label{result:innerDifference}
		\DZ(U) = \Zu(U)-\Zs(U) =
		\CInn e^{-iU} \Big(
		(0,0,1)^T + \chi(U) \Big)
	\end{equation}
	and, for $U \in \EInn$,
 \[
		| U^{\frac{7}{3}} \chi_1(U)| \leq \cttInnDiff, \qquad
		| U^{2} \chi_2(U) | \leq \cttInnDiff, \qquad
		| U \chi_3(U) | \leq \cttInnDiff.
\]
	{Moreover, $\CInn \neq 0$ if and only if $\DZ \neq 0$.}
\end{theorem}


\begin{remark}
In the paper \cite{articleInner}, the authors compute numerically the constant $\CInn$ which is approximately $\CInn\approx 1.63$. To compute it, it suffices to take into account that ${\CInn} = \lim_{\Im U \to - \infty}{\DY(U) e^{iU}}$. Then, we consider
 
	\begin{equation*}
		\CInn_{\rho} = \vabs{ \DY(-i\rho)} e^{\rho},
	\end{equation*} 
  
which, for $\rho$ big enough, satisfies $\CInn_{\rho} \approx \vabs{\CInn}$. 

Notice that, as we already claimed in Remark~\ref{rmk_before_ThInner}, we only need to analyze the difference on $\EInn$, contained in the imaginary axis.
\end{remark}

Theorem \ref{theorem:mainAnalytic} 
is proven in~\cite{articleInner} in two steps. First, for $\diamond = \unstable,\stable$, one proves the existence of the functions $\Zd$ in domains $\DdInn$ with $\kappa$ large enough. This is achieved through a fixed point argument of Perron type. Once the existence of the two functions in a common domain is proved, one looks for an equation for its difference that is used to derive the asymptotic formula~\eqref{result:innerDifference}. 


Next section specify the concrete steps for proving Theorem~\ref{thm:Stokes}. 

\section{Strategy to prove the main result}  
\label{sec:proofthmStokes}

To prove Theorem~\ref{thm:Stokes}, we follow  the strategy developed in~\cite{BCGS23} to prove that the Stokes constant associated to  unfoldings of the Hopf-zero singularity does not vanish. However, note that in~\cite{BCGS23} the strategy is tested on rather ``simple'' unfoldings of the singularity. On the contrary, in the present paper we deal with a given model of Celestial Mechanics, which requires more accurate estimates. 

The first step to prove Theorem~\ref{thm:Stokes} is to provide  a more quantitative version of Theorem~\ref{theorem:mainAnalytic}. It provides a larger domain of definition of the functions $\Zd$ given by Theorem \ref{theorem:mainAnalytic} and  relies on  carrying out a more detailed analysis of the two solutions $\Zd$ for $\diamond=\unstable,\stable$. 


\begin{theorem}\label{thm:FixedPointQuantitative}
The functions 	$\Zd(U) =(\Wd(U),\Xd(U),\Yd(U))^T$, $\diamond=\unstable,\stable$, introduced in Theorem \ref{theorem:mainAnalytic} are defined in $\mathcal{D}^\diamond_{\kappa^*}$ with (see \eqref{def:rho})
\begin{equation}\label{def:kappastar}
\kappa^*=6.24,\quad \gamma = \frac{1}{2} \quad \text{ and } \quad \rho^*=\rho(\kappa^*, \gamma) = \frac{\kappa^*}{\cos \gamma}<\rho_0=7.12.
\end{equation}
\end{theorem}


The proof of this theorem is deferred to Section \ref{sec:proof}. It follows the approach developed in \cite{BCGS23} for the inner equation associated to the Hopf-zero singularity.

The second step we perform to prove Theorem \ref{thm:Stokes} is to analyze the difference $\DZ$ (see~\eqref{result:innerDifference}).
This is provided by the next theorem, whose proof is computer assisted and is deferred to Section \ref{sec:difference}.

\begin{theorem}\label{thm:difference}
	The function $\DZ(U) = \Zu(U)-\Zs(U)$ introduced in \eqref{result:innerDifference} satisfies 
 \begin{equation}
\DZ(-i\rho)\neq 0, \label{eq:Delta-Z-at-kappa-star}
 \end{equation}
 for a $\rho>\rho_0$, where $\rho_0$ is the constant introduced in \eqref{def:kappastar}.
\end{theorem}

 By  Theorem \ref{theorem:mainAnalytic}, $\DZ(-i\rho)\neq 0$ implies that 
    \[\Theta\neq 0
    \]
and this completes the proof of Theorem \ref{thm:Stokes}.

\section{The domain of the solutions of the inner equation}
\label{sec:proof}
The proof of Theorem \ref{thm:FixedPointQuantitative} relies on a fixed point argument, and follows the same lines as the proof of the first part of Theorem \ref{theorem:mainAnalytic} in \cite{articleInner}. The main difference between the two proofs is that now we need explicit estimates for the fixed point argument. Moreover, they have to be rather accurate so that we obtain the $\kappa^*$ given by Theorem~\ref{thm:FixedPointQuantitative}. Note that a larger $\kappa^*$ would lead to a larger $\rho^*$ and $\rho_0$ (see \eqref{def:kappastar}), which would make harder to prove Theorem \ref{thm:difference}, since the difference $\Delta Z$ is exponentially small with respect to $\rho_0$.

%
We denote the components of all the functions and operators by a numerical sub-index $f=(f_1,f_2,f_3)^T$, unless stated otherwise.
%
Moreover, we deal only with the analysis for $\Zu$. The analysis for $\Zs$  is analogous and leads to exactly the same estimates.

\subsection{The fixed point equation and the functional setting}
\label{subsection:innerExistence}

%


The invariance equation~\eqref{eq:invariantEquationInner}  can be written as $\LL \Zu = \RRR[\Zu]$
where $\LL$ is  the linear operator
  
\begin{equation*}
	\LL \varphi =(\partial_U-\AAA)\varphi.
\end{equation*}
  
To construct a fixed point equation from \eqref{eq:invariantEquationInner}, we consider the following left inverse operator of $\LL$,
\begin{equation}\label{def:operatorGG}
		\GG[\varphi](U) = \left(\int_{-\infty}^0 \varphi_1(s+U) ds, 
		\int_{-\infty}^0 e^{-i s} \varphi_2(s+U) ds , 
		\int_{-\infty}^0 e^{i s} \varphi_3(s+U) ds \right)^T,
	\end{equation}
(see Lemma \ref{lemma:boundsOperatorGG} below). We look for a fixed point of the  operator 
\begin{equation}\label{def:operatorFF}
	\FF= \GG \circ \RRR,
\end{equation}
 in a suitable Banach space. 



Given $\nu \in \reals$ and $\kInn>0$, we define the norm
\[
	\normInn{\varphi}_{\nu}= \sup_{U \in \DuInn} \vabs{U^{\nu} \varphi(U) },	
	%
	%
\]
where the domain $\DuInn$ is given in~\eqref{def:domainInnner}, 
and we introduce the Banach space
\begin{equation}\label{def:banachSpacesX}
	\begin{split}
		\XcalU_{\nu}&= 
		\left\{ \varphi: \DuInn \to \complexs  \st  
		\varphi \text{ analytic, } 
		\normInn{\varphi}_{\nu} < +\infty \right\}.
	\end{split}
\end{equation}

	%
	
%
%
A solution of 
$
\Zu=\FF[\Zu]
$
belonging to $\XcalU_{\eta} \times \XcalU_{\nu} \times \XcalU_{\nu}$
with $\eta,\nu>0$
satisfies equation~\eqref{eq:invariantEquationInner} and the asymptotic condition~\eqref{eq:asymptoticConditionsInner}.
Then, to prove Theorem \ref{thm:FixedPointQuantitative}, we look for a fixed point of  the operator $\FF$  in the Banach space
\begin{equation}\label{def:XcalUTotal}
	\XcalUTotal = \XcalU_{\frac{8}{3}} \times \XcalU_{\frac{4}{3}} \times \XcalU_{\frac{4}{3}}, 
\end{equation}
endowed with the norm
\begin{equation}
\label{def:norminnTotal}
\normInnTotal{\varphi}= \max\left\{
\normInn{\varphi_1}_{\frac{8}{3}},
\normInn{\varphi_2}_{\frac{4}{3}}, 
\normInn{\varphi_3}_{\frac{4}{3}}\right\}.    
\end{equation}

Theorem \ref{thm:FixedPointQuantitative} is a direct consequence of the following proposition.
\begin{proposition}\label{proposition:innerExistence}For any $\kInn \geq \kInnS$ where $\kInnS$ is the constant introduced in \eqref{def:kappastar}, the fixed point equation $\Zu = {\FF}[\Zu]$ has a solution $\Zu \in \XcalUTotal$.
	%
\end{proposition}

We devote the rest of this section to prove Proposition \ref{proposition:innerExistence}. Note that our goal is to compute explicit  estimates throughout the proof to obtain a ``not too large'' $\kInnS$. 

%

Let us  explain the main steps of the proof, which are carried out in the forthcoming sections.

\begin{itemize}
	\item In Section~\ref{subsection:operatorGG}, we provide properties of the Banach spaces introduced in \eqref{def:banachSpacesX} and give explicit bounds for the norm of the linear operator $\GG$ in \eqref{def:operatorGG}.
	\item In Section~\ref{subsection:firstIteration}, we give estimates for $\FF[0]$. 
	\item In Section~\ref{subsection:LipschitzRRR}, we provide estimates for the derivatives of $\RRR$ (see \eqref{def:operatorRRRInner}) in a suitable domain.
	\item In Section~\ref{subsection:operatorFF}, we provide explicit expressions for the Lipschitz constant of the operator $\FF$ in a suitable region of the Banach space~\eqref{def:XcalUTotal}. Finally, we prove that the Lipschitz constant is smaller than $1$, and therefore  $\FF$ is contractive and has a unique fixed point.
\end{itemize}

To have accurate estimates in these steps, we rely on explicit formulae for the non-linear operator $\RRR$ in  \eqref{def:operatorRRRInner} and its first and second derivatives. They are given in  Appendix~\ref{appendix:formulasRRR}.

Note that in \cite{articleInner}, we were not able to prove that $\FF$ is contractive. Instead, we considered a  slightly modified operator which had the same fixed points as $\FF$ and a smaller Lipschitz constant. Since the estimates in the present paper are more accurate, we are able to prove that the original operator $\FF$ is contractive.


\subsection{Banach space  properties and the integral  operator}
\label{subsection:operatorGG}
Throughout this and the forthcoming sections we use without mentioning the following remark and lemma.

\begin{remark}
  Let $U \in \DuInn$, then
  $
    \vabs{U} \geq \kInn.
  $
\end{remark}

Next lemma, proven in~\cite{Bal06}, gives some properties of the Banach spaces $\XcalU_{\eta}$ introduced in \eqref{def:banachSpacesX}.
%
%
%

\begin{lemma}\label{lemma:sumNorms}
	Let $\kInn>0$ and $\nu, \eta \in \reals$. The following statements hold:
	\begin{enumerate}
		\item If $\nu > \eta$, then $\XcalU_{\nu} \subset \XcalU_{\eta}$ and
		$
		\normInn{\varphi}_{\eta} \leq  
		{\kInn}^{\eta-\nu} 
		\normInn{\varphi}_{\nu}.
		$
		\item If $\varphi \in \XcalU_{\nu}$ and 
		$\zeta \in \XcalU_{\eta}$, then the product
		$\varphi \zeta \in \XcalU_{\nu+\eta}$ and
		$
		\normInn{\varphi \zeta}_{\nu+\eta} \leq \normInn{\varphi}_{\nu} \normInn{\zeta}_{\eta}.
		$
	\end{enumerate}
\end{lemma}

Now we provide estimates for the integral operator introduced in~\eqref{def:operatorGG}. 

\begin{lemma}\label{lemma:boundsOperatorGG}
Consider the linear operator $\GG[\phiA]=(\GG_1[\phiA_1],\GG_2[\phiA_2],\GG_3[\phiA_3])^T$ as defined in~\eqref{def:operatorGG} and fix $\eta>1$, $\nu>0$ and  $\kInn\geq 1$.
Then, $\GG: \XcalU_{\eta} \times \XcalU_{\nu} \times \XcalU_{\nu} \to \XcalU_{\eta-1} \times \XcalU_{\nu} \times \XcalU_{\nu}$ is a continuous linear operator and is a left-inverse of $\LL$.

Moreover,
 \begin{enumerate}
	\item For $j=1,2,3$ and $\eta>1$, the linear operator  
	$\GG_j: \XcalU_{\eta} \to \XcalU_{\eta-1}$ is continuous and satisfies that, for $\phiA \in \XcalU_{\eta}$,
 \[
		\normInn{\GG_j[\phiA]}_{\eta-1} \leq G_{\eta} \normInn{\phiA}_{\eta},
		\qquad
		G_{\eta} = \frac{\sqrt{\pi} \, \Gamma\paren{\frac{\eta-1}2 }}{2 \Gamma\paren{\frac{\eta}2}}.
	\]
	\item For $j=2,3$ and $\nu>0$, the linear operator  
	$\GG_j: \XcalU_{\nu} \to \XcalU_{\nu}$ is continuous and, for all $0<\sigma\leq\gamma$ (see \eqref{def:domainInnner}) and $\phiA \in \XcalU_{\nu}$, satisfies that 
 \[
		\normInn{\GG_j[\phiA]}_{\nu} \leq 
		\frac{\normInn{\phiA}_{\nu}}{\sin{\sigma}(\cos{\sigma})^{\nu}}.
	\]
	\end{enumerate}
\end{lemma}

The proof of this lemma follows the same lines of  the proof of \cite[Lemma 2.1]{BCGS23}. We first state the following lemma.

\begin{lemma}\label{lemma:boundsIntegralOperatorGG}
Let $\eta>1$. Then,
\[
	\int^{0}_{-\infty} \frac{d s}{(1+s^2)^{\frac{\eta}2}}
	\leq
	\frac{\sqrt{\pi} \, \Gamma\paren{\frac{\eta-1}2 }}{2 \Gamma\paren{\frac{\eta}2}}.
\]
\end{lemma}
\begin{proof}
Considering the change of coordinates $s=\tan x$,
\[
		\int^{0}_{-\infty} \frac{d s}{(1+s^2)^{\frac{\eta}2}} \leq 
		\int_{-\frac{\pi}2}^0 \frac{dx}{(1+\tan(x)^2)^{\frac{\eta}2}\cos^2(x)}
		= 
		\int^{\frac{\pi}2}_{0} \cos^{\eta-2}(x) dx.
\]
To complete the proof of the lemma, it only remains to recall that 
\[
\int^{\frac{\pi}2}_{0} \cos^{\eta-2}(x) dx=\frac{1}{2}B\left(\frac{\eta-1}{2},\frac{1}{2}\right)=\frac{\sqrt{\pi} \, \Gamma\paren{\frac{\eta-1}2 }}{2 \Gamma\paren{\frac{\eta}2}},
\]
where $B$ is the classical beta function.
\end{proof}

\begin{proof}[Proof of Lemma \ref{lemma:boundsOperatorGG}]
	%
	To prove the first item of the lemma, we consider $\eta>1$ and $\phiA \in \XcalU_{\eta}$.
	Then, for $j=1,2,3$ and $U \in \DuInn$, one has that
	\[
		\vabs{\GG_j[\phiA](U)} \leq
		\normInnU{\phiA}_{\eta} \int^{0}_{-\infty} \frac{dS}{\vabs{s+U}^{\eta}}.	
	\]
Moreover, for $s\in(-\infty,0]$, one has that $\vabs{s+U}^2 \geq s^2+\vabs{U}^2$. Then, by Lemma~\ref{lemma:boundsIntegralOperatorGG},
	\[
		\vabs{\GG_j[\phiA](U)} \leq 
		 \int^{0}_{-\infty} \frac{\normInnU{\phiA}_{\eta} }{(s^2+\vabs{U}^2)^{\frac{\eta}2}} ds =
		\frac{\normInnU{\phiA}_{\eta}}{\vabs{U}^{\eta-1}} 
		\int_{-\infty}^0 \frac{ds}{(1+s^2)^{\frac{\eta}2}}
		\leq 
		\frac{\normInnU{\phiA}_{\eta}}{\vabs{U}^{\eta-1}} 
		\frac{\sqrt{\pi} \, \Gamma\paren{\frac{\eta-1}2 }}{2 \Gamma\paren{\frac{\eta}2}}.
	\] 
	
	Next, we prove the second item.
	By the geometry of $\DuInn$ and using the Cauchy's theorem, we can change the path of integration in the integral \eqref{def:operatorGG} to $t e^{i {\sigma}}$, $t \in (-\infty, 0]$, with $0\leq {\sigma} \leq \gamma$.
	Then, for $\nu>0$ and $\phiA \in \XcalU_{\nu}$,
	\[
		{\GG_2[\phiA](U)} = 
		\int_{-\infty}^0 e^{-i t \cos{\sigma}}  e^{t\sin{\sigma}}
		\varphi(U+t e^{i {\sigma}}) e^{i{\sigma}} dt.
	\]
	Notice that $U+t e^{i \sigma} \in \DuInn$. Then, the function $f(t) = |U+ te^{i{\sigma}}|$ has a minimum at ${t_*=-\vabs{U}\cos(\arg U -{\sigma})}$ and $f(t_*)=\vabs{U\sin(\arg U -{\sigma})}$. 
    Notice that, since $\arg U \in [\frac{\pi}2,\frac{3\pi}2]$ and $t_*$ is negative only if $\arg U \in [\frac{\pi}2,\frac{\pi}2+\sigma]$, one has that
    \[
    \vabss{U+t e^{i {\sigma}}} \geq \vabs{U} \sin\paren{\frac{\pi}2-{\sigma}}=\vabs{U}\cos{\sigma}.
    \]
	%
	Therefore,
	\[
		\vabs{\GG_2[\phiA](U)} \leq \frac{\normInnU{\phiA}_{\nu}}{\vabs{U}(\cos\sigma)^{\nu}}
		\int_{-\infty}^0 e^{t\sin\sigma} dt
		= \frac{\normInnU{\phiA}_{\nu}}{\vabs{U}\sin\sigma(\cos\sigma)^{\nu}}.
	\]
	The norm for $\GG_3[\phiA]$ can be treated in the same way changing the integration path to $t e^{-i\sigma}$.
\end{proof}

\subsection{Estimates for $\FF(0)$}
\label{subsection:firstIteration}
The next proposition gives estimates for  $\FF(0)$.

\begin{proposition}\label{proposition:firstIteration}
Fix  $\kappa\geq1$. Then, the operator $\FF$ in \eqref{def:operatorFF} satisfies
\[
	\normInn{\FF_1[0]}_{\frac83}  
	\leq  \alpha_0(\kappa), 
	\qquad
	\normInn{\FF_j[0]}_{\frac43} \leq \beta_0(\kappa),\quad j=2,3 
\]
and
	\[
		\normInnTotal{\FF[0]} = \max \{ \alpha_0(\kappa), \beta_0(\kappa)\},
	\]
where $\alpha_0(\kappa)$ and $\beta_0(\kappa)$ are decreasing functions satisfying that
\[
\begin{split}
	\alpha_0(\kappa) 
	&=
	\frac{8}{243} + \frac{32\sqrt{\pi}\, \Gamma(\frac73)}{729 \, \Gamma(\frac{17}6)} \frac1{\kappa^2} \paren{
	\frac{16}{81(2 - \zeta_0(\kappa))} +
	\frac{\zeta_3(\kappa)}{8-\zeta_2(\kappa)} +
	\frac{16}{81\kappa^2} \frac{\zeta_3(\kappa)}{16-\zeta_4(\kappa)}},
	\\
	\beta_0(\kappa)
	&=
	\frac29 + 
	\frac{14\sqrt{\pi}\,\Gamma(\frac23)}{81\, \Gamma(\frac76)} +
	\paren{
	\frac{\sqrt{\pi} \, \Gamma(\frac76)}{9\,\Gamma(\frac53)} \frac1{\kappa} +
	\frac{2\sqrt{\pi}\, \Gamma(\frac53)}{81 \Gamma(\frac{13}6)}  \frac1{\kappa^2}} 
	\frac{\zeta_1(\kappa)}{2-\zeta_0(\kappa)},
\end{split}
\]
with
\[
\begin{split}
	\zeta_0(\kappa) &= \frac{16}{81\kappa^2} 
	\paren{1+3\sqrt{1+\frac{16}{81\kappa^2}}\,},
	\qquad
	\zeta_1(\kappa) = \frac{16}{81} 
	\paren{1+3\sqrt{1+\frac{16}{81\kappa^2}}\, }, 
	\\
	\zeta_2(\kappa) &= \frac{16}{81\kappa^2}
	\paren{6\sqrt{1+{\frac{16}{81\kappa^2}}}+8+{\frac{128}{81\kappa^2}}},
	\\
	\zeta_3(\kappa) &= \frac{16}{81} \paren{ \frac3{2\sqrt{1-\frac{16}{81\kappa^2}}} + 6 \sqrt{1+ \frac{16}{81\kappa^2} } + 6 + \frac{128}{81\kappa^2}},
	\\
	\zeta_4(\kappa) &=
	\frac{32}{81\kappa^2} \paren{10 + \frac{256}{81\kappa^2} + 
		\frac{2^9}{3^7 \kappa^4} 
		+ \sqrt{1+\frac{16}{81\kappa^2}}\paren{12 + \frac{368}{81\kappa^2} +\frac{2^9 7}{3^8\kappa^4}} }.
\end{split}
\]
\end{proposition}

	%


We devote the rest of the section to prove Proposition \ref{proposition:firstIteration}. We first provide estimates for $\RRR[0]$.
\begin{lemma} \label{lemma:expressionRRR0}
Let the operator $\RRR$ be as defined in~\eqref{def:operatorRRRInner}. 
Then, 
\[
\begin{split}
	\RRR_1[0](U) &=
	\frac{2^6}{3^6  U^{\frac{11}3}} \paren{1+R_1^0(U)},
	\\
	\RRR_2[0](U) 
	&=
	\paren{-\frac{2}{9 U^{\frac43}} -
		\frac{4i}{81 U^{\frac73}} } \paren{1+R_{2,3}^0(U)},
	\\
	\RRR_3[0](U) 
	&=
	\paren{-\frac{2}{9 U^{\frac43}}+
		\frac{4i}{81 U^{\frac73}} } \paren{1+R_{2,3}^0(U)},
\end{split}
\]
where
\[
\begin{split}
	R_1^0(U) &= \frac34
	\paren{
		\frac2{3\sqrt{1+\JJ_0}(1+\sqrt{1+\JJ_0})}
		+\frac1{(1+\JJ_0)^{\frac{3}{2}}}
	}
	\paren{\frac{2(1+\JJ_0)^{\frac32}}{2(1+\JJ_0)^{\frac32}-\JJ_0}}-1,
	\\
	R_{2,3}^0(U) &= \frac2{2(1+\JJ_0)^{\frac32}-\JJ_0}-1,
\end{split}
\]
with $\JJ_0 =\JJ_0(U) = \frac{16}{81 U^2}$.
\end{lemma}

\begin{proof}
By the expression of $\RRR$ in \eqref{def:operatorRRRInner}, one sees that	
\begin{equation}\label{proof:operatorRRRStokes}
	\RRR[0](U) = \paren{\frac{-\partial_U \KK(U,0)}{1+\partial_W \KK(U,0)},\,
		\frac{i\partial_Y \KK(U,0)}{1+\partial_W \KK(U,0)}, \,
		\frac{-i\partial_X \KK(U,0)}{1+\partial_W \KK(U,0)} }.
\end{equation}
Notice that, by the definition of $\JJ$ in \eqref{def:hFunction}, one has that
$
	\JJ_0(U) := \JJ(U,0) = \frac{16}{81 U^2}.
$	
Moreover, using the formulae in Appendix~\ref{appendix:formulasRRR} for function $\JJ$, one sees that
  
\begin{align*}
	\partial_U \JJ(U,0) &=
	-\frac{32}{81 U^{3}},
	&
	\partial_W \JJ(U,0) &=
	-\frac{16}{27 U^{\frac{4}{3}}},
	\\
	\partial_X \JJ(U,0) &=
	-\frac{4i}{3 U^{\frac{2}{3}}}
	-\frac8{27U^{\frac{5}{3}}},
	&
	\partial_Y \JJ(U,0) &=
	\frac{4i}{3 U^{\frac{2}{3}}}
	-\frac8{27U^{\frac{5}{3}}}.
\end{align*}
  
In the same way, using the formulae in Appendix~\ref{appendix:formulasRRR}, one obtains that
\[
\begin{split}
	\partial_U \KK(U,0) &=
	-\frac{2^5}{3^6 U^{\frac{11}3}}  \frac1{\sqrt{1+\JJ_0}(1+\sqrt{1+\JJ_0})}
	-\frac{2^4}{3^5 U^{\frac{11}3}} \frac1{(1+\JJ_0)^{\frac{3}{2}}},
	\\
	&= 
	- \frac{2^4}{3^5  U^{\frac{11}3}}
	\paren{
		\frac2{3\sqrt{1+\JJ_0}(1+\sqrt{1+\JJ_0})}
		+\frac1{(1+\JJ_0)^{\frac{3}{2}}}
	},
	\\
	\partial_W \KK(U,0) &=
	-\frac8{81 U^{2}}
	\frac1{(1+\JJ_0)^{\frac32}}
	= -\frac{\JJ_0}{2(1+\JJ_0)^{\frac32}} ,
	\\
	\partial_X \KK(U,0) &=
	\paren{-\frac{2i}{9 U^{\frac43}} -
		\frac4{81 U^{\frac73}} }
	\frac1{(1+\JJ_0)^{\frac32}},
	\\
	\partial_Y \KK(U,0) &=
	\paren{\frac{2i}{9 U^{\frac43}} -
		\frac4{81 U^{\frac73}} }
	\frac1{(1+\JJ_0)^{\frac32}}.
\end{split}
\]
Finally, applying these expressions to \eqref{proof:operatorRRRStokes}, 
\[
\begin{split}
	\RRR_1[0](U) &= 
	\frac{2^4}{3^5  U^{\frac{11}3}}
	\paren{
		\frac2{3\sqrt{1+\JJ_0}(1+\sqrt{1+\JJ_0})}
		+\frac1{(1+\JJ_0)^{\frac{3}{2}}}
	}
	\paren{\frac{2(1+\JJ_0)^{\frac32}}{2(1+\JJ_0)^{\frac32}-\JJ_0}},
	\\[0.8em]
	\RRR_2[0](U) &=
	\paren{-\frac{2}{9 U^{\frac43}} -
		\frac{4i}{81 U^{\frac73}} }
	\frac2{2(1+\JJ_0)^{\frac32}-\JJ_0},
	\\[0.8em]
	\RRR_3[0](U) &=
	\paren{-\frac{2}{9 U^{\frac43}} +
		\frac{4i}{81 U^{\frac73}} }
	\frac2{2(1+\JJ_0)^{\frac32}-\JJ_0}.
\end{split}
\]
\end{proof}

Notice that, for $\vabs{U} \to \infty$, 
\[
\begin{split}
	R_{1}^0(U) &= 
	- \frac{13}{16}\JJ_0 + \frac{13}{32}\JJ_0^2  + \OO(\JJ_0^3)
	= - \frac{13}{81 U^2} + \frac{104}{6561 U^4} + \OO\paren{\frac1{U^6}},
	\\
	R_{2,3}^0(U) &= -\JJ_0+\frac58 \JJ_0^2 + \OO(\JJ_0^3)
	= - \frac{16}{81 U^2} + \frac{160}{6561 U^4} + \OO\paren{\frac1{U^6}}.
\end{split}
\]
Next lemma provides explicit estimates for these functions in terms of the parameter~$\kappa$.

\begin{lemma}\label{lemma:constantsRRR0}
For $\kappa\geq1$ and $U \in \DuInn$ (see \eqref{def:domainInnner}) one has that
\[
	\vabs{R_{1}^0(U)} \leq \frac{C_{1}^0(\kappa)}{\vabs{U}^2},
	\qquad
	\vabs{R_{2,3}^0(U)} \leq \frac{C_{2,3}^0(\kappa)}{\vabs{U}^2},
\]
where $C_1^0(\kappa)$ and $C_{2,3}^0(\kappa)$ are decreasing functions satisfying that
\[
	C_{2,3}^0(\kappa) =\frac{\zeta_1(\kappa)}{2-\zeta_0(\kappa)},
		\qquad
	C_1^0(\kappa) = \frac{\zeta_3(\kappa)}{8-\zeta_2(\kappa)} + \frac{16}{81(2 - \zeta_0(\kappa))}
	+ \frac{16}{81\kappa^2} \frac{\zeta_3(\kappa)}{16-\zeta_4(\kappa)},
\]
with
\[
\begin{split}
	\zeta_0(\kappa) &= \frac{16}{81\kappa^2} 
	\paren{3\sqrt{1+\frac{16}{81\kappa^2}}+1}\,,
	\qquad
	\zeta_1(\kappa) = \frac{16}{81} 
	\paren{3\sqrt{1+\frac{16}{81\kappa^2}}+1\, }, 
	\\
	\zeta_2(\kappa) &= \frac{16}{81\kappa^2}
	\paren{6\sqrt{1+{\frac{16}{81\kappa^2}}}+8+{\frac{128}{81\kappa^2}}},
	\\
	\zeta_3(\kappa) &= \frac{16}{81} \paren{ \frac3{2\sqrt{1-\frac{16}{81\kappa^2}}} + 6 \sqrt{1+ \frac{16}{81\kappa^2} } + 6 + \frac{128}{81\kappa^2}},
	\\
	\zeta_4(\kappa) &=
	\frac{32}{81\kappa^2} \paren{10 + \frac{256}{81\kappa^2} + 
		\frac{2^9}{3^7 \kappa^4} 
		+ \sqrt{1+\frac{16}{81\kappa^2}}\paren{12 + \frac{368}{81\kappa^2} +\frac{2^9 7}{3^8\kappa^4}} }.
\end{split}
\]
\end{lemma}

\begin{proof}
First, we compute $C_{2,3}^{0}(\kappa)$. By Lemma~\ref{lemma:expressionRRR0}, 
\[
	R_{2,3}^0(U) = \frac{2+\JJ_0-2(1+\JJ_0)^{\frac32}}{2+ \claus{
			2(1+\JJ_0)^{\frac32}-2-\JJ_0}}
	= \frac{r_1(\JJ_0)}{2 + r_0(\JJ_0)}.
\]
Then, by the mean value theorem, 
\[
\begin{split}
	\vabs{r_0(\JJ_0)} &\leq \vabs{\JJ_0}
	\sup_{\sigma\in[0,1]} 
	\vabs{3\sqrt{1+\sigma\JJ_0}-1}
	\leq \vabs{\JJ_0} 
\paren{3\sqrt{1+\vabs{\JJ_0}}+1},
	\\
	\vabs{r_1(\JJ_0)} &\leq \vabs{\JJ_0} 
	\sup_{\sigma\in[0,1]}\vabs{1-3\sqrt{1+\sigma\JJ_0}} 
	\leq \vabs{\JJ_0}\paren{1+3\sqrt{1+\vabs{\JJ_0}}}.
\end{split}
\]
Since, $U \in \DuInn$, one has that $\vabs{U} \geq \kappa$. 
Then, $\vabs{\JJ_0(U)} \leq \frac{16}{81 \kappa^2} $ and, as a result,
\[
\begin{split}
     \vabs{r_0(\JJ_0)} &\leq\frac{16}{81\kappa^2} 
     \paren{3\paren{1+\frac{16}{81\kappa^2}}^{\frac12}+1}= \zeta_0(\kappa),
     \\
     {\vabs{r_1(\JJ_0){U}^2}} &\leq  \frac{16}{81} 
     \paren{3\paren{1+\frac{16}{81\kappa^2}}^{\frac12}+1}=\zeta_1(\kappa).
\end{split}
\]
Notice that $\zeta_0(\kappa)$ and $\zeta_1(\kappa)$ are decreasing functions for positive $\kappa$ and $\zeta_0(1)<2$. Therefore, applying the triangular inequality, we denote
\[
    C^0_{2,3}(\kappa) = \frac{\zeta_1(\kappa)}{2-\zeta_0(\kappa)},
\]
which, by construction, is  a decreasing function as well.

Analogously, we compute $C_1^0(\kappa)$. By Lemma~\ref{lemma:expressionRRR0},
\[
\begin{split}
	R_1^0(U) =& \frac34
	\paren{
		\frac2{3\sqrt{1+\JJ_0}(1+\sqrt{1+\JJ_0})}
		+\frac1{(1+\JJ_0)^{\frac{3}{2}}}
	}
	\paren{\frac{2(1+\JJ_0)^{\frac32}}{2(1+\JJ_0)^{\frac32}-\JJ_0}}-1
	\\
	=&
	\frac34
	\paren{
		\frac{2(1+\JJ_0)+3(1+\sqrt{1+\JJ_0})}{3(1+\JJ_0)^{\frac32}(1+\sqrt{1+\JJ_0})}  }
	\paren{\frac{2(1+\JJ_0)^{\frac32}}{2(1+\JJ_0)^{\frac32}-\JJ_0}}-1
	\\
	=&
	\frac34\paren{ \frac43 +
		\frac{2(1+\JJ_0)+3(1+\sqrt{1+\JJ_0}) -4(1+\JJ_0)^{\frac32}(1+\sqrt{1+\JJ_0})}{3(1+\JJ_0)^{\frac32}(1+\sqrt{1+\JJ_0})}  }
	\\
	&\paren{1+\frac{2(1+\JJ_0)^{\frac32} -2(1+\JJ_0)^{\frac32}+\JJ_0}{2(1+\JJ_0)^{\frac32}-\JJ_0}}-1
\end{split}
\]
and
\[
\begin{split}
		R_1^0(U) =& 
	\paren{1 +
		\frac{1+3\sqrt{1+\JJ_0} -4(1+\JJ_0)^{\frac32} -6\JJ_0-4\JJ_0^2}{4(1+\JJ_0)^{\frac32}(1+\sqrt{1+\JJ_0})}  }
	\paren{1+\frac{\JJ_0}{2(1+\JJ_0)^{\frac32}-\JJ_0}}-1  
	\\
	=& \, 
	\frac{1+3\sqrt{1+\JJ_0} -4(1+\JJ_0)^{\frac32} -6\JJ_0-4\JJ_0^2}{4(1+\JJ_0)^{\frac32}(1+\sqrt{1+\JJ_0})}
	+\frac{\JJ_0}{2(1+\JJ_0)^{\frac32}-\JJ_0} 
	\\
	&+\JJ_0 \frac{1+3\sqrt{1+\JJ_0} -4(1+\JJ_0)^{\frac32} -6\JJ_0-4\JJ_0^2}{4(1+\JJ_0)^{\frac32}(1+\sqrt{1+\JJ_0})\paren{2(1+\JJ_0)^{\frac32}-\JJ_0}}
	\\
	=& \,
	\frac{r_3(\JJ_0)}{8+r_2(\JJ_0)} + \frac{\JJ_0}{2 + r_0(\JJ_0)}
	+ \JJ_0 \frac{r_3(\JJ_0)}{16+r_4(\JJ_0)}.
\end{split}
\]
Then, the function 
\[
r_2(\JJ_0) = 4(1+\JJ_0)^{\frac32}(1+\sqrt{1+\JJ_0})-8
	= 4(1+\JJ_0)^{\frac32} -4 + 8\JJ_0 + 4\JJ_0^2, 
 \]
by the mean value theorem and taking into account that  $\vabs{\JJ_0(U)} \leq \frac{16}{81 \kappa^2}$, satisfies
\[
\begin{split}
		\vabs{r_2(\JJ_0)} &\leq \vabs{\JJ_0} \sup_{\sigma \in [0,1]}
	\vabs{6\sqrt{1+\sigma\JJ_0}+8+8\sigma\JJ_0}
	\leq \vabs{\JJ_0} 
	\paren{6\sqrt{1+\vabs{\JJ_0}}+8+8\vabs{\JJ_0}}, 
	\\
	 &\leq \frac{16}{81\kappa^2}
	\paren{6\sqrt{1+{\frac{16}{81\kappa^2}}}+8+{\frac{128}{81\kappa^2}}}= \zeta_2(\kappa).
\end{split}
\]
Notice that,  $\vabs{\JJ_0(U)} \leq \frac{16}{81\kappa^2}<1$ for $\kappa\geq1$. Then, the functions
\[
\begin{split}	r_3(\JJ_0) =& \, 1+3\sqrt{1+\JJ_0} -4(1+\JJ_0)^{\frac32} -6\JJ_0-4\JJ_0^2,
\\
r_4(\JJ_0) =& \,  4(1+\JJ_0)^{\frac32}(1+\sqrt{1+\JJ_0})\paren{2(1+\JJ_0)^{\frac32}-\JJ_0}-16 \\
	=& \,  4(-2+5\JJ_0+4\JJ_0^2+\JJ_0^3) 
	+4(1+\JJ_0)^{\frac32}(2+3\JJ_0+2\JJ_0^2),
\end{split}
\]
 satisfy
\[
\begin{split}	
	\vabs{r_3(\JJ_0)} &\leq \vabs{\JJ_0} \sup_{\sigma\in[0,1]}
	\vabs{\frac3{2\sqrt{1+\sigma\JJ_0}} - 6\sqrt{1+\sigma\JJ_0} -6 - 8\sigma \JJ_0}
	\\
	&\leq \vabs{\JJ_0} \paren{ \frac3{2\sqrt{1-\vabs{\JJ_0}}} + 6 \sqrt{1+\vabs{\JJ_0}} + 6 + 8\vabs{\JJ_0}}, 
	\\
	\vabs{r_3(\JJ_0){U}^2} &\leq \frac{16}{81} \paren{ \frac3{2\sqrt{1-\frac{16}{81\kappa^2}}} + 6 \sqrt{1+ \frac{16}{81\kappa^2} } + 6 + \frac{128}{81\kappa^2}}= \zeta_3(\kappa)
\end{split}	
\]
and
\[
\begin{split}	
		\vabs{r_4(\JJ_0)} &\leq 4\vabs{\JJ_0} \sup_{\sigma \in [0,1]}
	\vabs{5+8\sigma\JJ_0+3\sigma^2\JJ_0^2+
		\frac12\sqrt{1+\sigma\JJ_0}(12+23\sigma\JJ_0+14\sigma^2\JJ_0^2)} 
	\\
	&\leq  2\vabs{\JJ_0} \paren{10 + 16 \vabs{\JJ_0} + 6\vabs{\JJ_0}^2 
		+ \sqrt{1+\vabs{\JJ_0}}\paren{12 + 23\vabs{\JJ_0} +14\vabs{\JJ_0}^2} }
	\\
	&\leq 
	\frac{32}{81\kappa^2} \paren{10 + \frac{256}{81\kappa^2} + 
		\frac{2^9}{3^7 \kappa^4} 
		+ \sqrt{1+\frac{16}{81\kappa^2}}\paren{12 + \frac{368}{81\kappa^2} +\frac{2^9 7}{3^8\kappa^4}} }=\zeta_4(\kappa).
  \end{split}	
\]
Notice that $\zeta_0(\kappa), \zeta_2(\kappa)$ and $\zeta_4(\kappa)$ are decreasing functions for positive $\kappa$ and one can easily checked  that $\zeta_0(1)<2$, $\zeta_2(1)<8$ and $\zeta_4(1)<16$. Then, we denote
\[
	C_1^0(\kappa) = \frac{\zeta_3(\kappa)}{8-\zeta_2(\kappa)} + \frac{16}{81}\frac1{2 - \zeta_0(\kappa)}
	+ \frac{16}{81\kappa^2} \frac{\zeta_3(\kappa)}{16-\zeta_4(\kappa)},
\]
which, by construction, is a decreasing function for $\kappa\geq 1$.
\end{proof}


\begin{proof}[Proof of Proposition~\ref{proposition:firstIteration}]
Let us recall that $\FF=\GG\circ\RRR$. Then, by Lemmas~\ref{lemma:boundsIntegralOperatorGG}, \ref{lemma:expressionRRR0} and~\ref{lemma:constantsRRR0}, and proceeding as in the proof of Lemma~\ref{lemma:boundsOperatorGG},
\[
\begin{split}	
	\vabs{\FF_1[0](U)} 
	&= \vabs{\int_{-\infty}^0 \RRR_1[0](s+U) ds} 
	\leq \frac{2^6}{3^6} \vabs{ \int_{-\infty}^0 \frac{ds}{(s+U)^{\frac{11}3}}}
	+ \frac{2^6}{3^6} \vabs{ \int_{-\infty}^0 \frac{R_1(s+U)}{(s+U)^{\frac{11}3}} ds} 
	\\[0.8em]
	&\leq  \frac{2^6}{3^6} \vabs{\frac3{8 U^{\frac83}}}
	+\frac{2^6}{3^6}  C_{1}^0(\kappa) \int_{-\infty}^0 \frac{ds}{\vabs{s+U}^{\frac{17}3}}
	\leq \frac{2^3}{3^5 \vabs{U}^{\frac83}} 
	+\frac{2^6}{3^6}  C_{1}^0(\kappa) \int_{-\infty}^0 \frac{ds}{(s^2+\vabs{U}^2)^{\frac{17}6}}
	\\[0.8em]
	&\leq  \frac{8}{3^5 \vabs{U}^{\frac83}} + 
	\frac{2^6  C_{1}^0(\kappa) }{3^6 \vabs{U}^{\frac{14}3}}  \int_{-\infty}^0 
	\frac{ds}{(1+s^2)^{\frac{17}6}}
	\leq  \frac{8}{3^5 \vabs{U}^{\frac83}} + 
	\frac{2^6  C_{1}^0(\kappa) }{3^6 \vabs{U}^{\frac{14}3}}  \frac{\sqrt{\pi}\, \Gamma(\frac73)}{2 \Gamma(\frac{17}6)}.
 \end{split}	
\]
Then
\[
\begin{split}
	\normInn{\FF_1[0]}_{\frac83}  
	&\leq  \frac{8}{243} + \frac{32\sqrt{\pi}\, \Gamma(\frac73)}{729 \, \Gamma(\frac{17}6)} \frac{C_1^0(\kappa)}{\kappa^2}.
\end{split}
\]
Following the same ideas, one has that 
\[
\begin{split}
	&\vabs{\FF_2[0](U)} 
	= \vabs{\int_{-\infty}^0 e^{-is} \RRR_2[0](s+U) ds} 
	\\
	&\leq 
	\frac29 \vabs{\int_{-\infty}^0 \frac{e^{-is}}{(s+U)^{\frac43}} ds } +
	\frac29 \vabs{\int_{-\infty}^0 \frac{e^{-is}R_{2,3}(s+U)}{(s+U)^{\frac43}} ds }
	+ \frac4{81} \vabs{\int_{-\infty}^0 \frac{e^{-is}(1+R_{2,3}(s+U))}{(s+U)^{\frac73}} ds }.
\end{split}
\]
Notice that the first integral in the inequality satisfies that
\[
\begin{split}
	\int_{-\infty}^0 \frac{e^{-is}}{(s+U)^{\frac43}} ds  
	&=
	i\boxClaus{\frac{e^{-is}}{(s+U)^{\frac43}}}_{s=-\infty}^{s=0}
	+ \frac{4i}3 \int_{-\infty}^0 \frac{e^{-is}}{(s+U)^{\frac73}} ds\\
	&= \frac{i}{U^{\frac43}} 
	+ \frac{4i}3 \int_{-\infty}^0 \frac{e^{-is}}{(s+U)^{\frac73}} ds.
\end{split}
\]
Then, proceeding as in the proof of Lemma~\ref{lemma:boundsOperatorGG},
\[
\begin{split}
	\vabs{\FF_2[0](U)}  \leq&
	\frac{2}{9\vabs{U}^{\frac43}} +
	\frac{8}{27} \int_{-\infty}^0 \frac{ds}{\vabs{s+U}^{\frac73}} +
	\frac{2 C_{2,3}^0(\kappa)}9 \int_{-\infty}^0 \frac{ds}{\vabs{s+U}^{\frac{10}3}} \\
	&+ \frac{4}{81} \int_{-\infty}^0 \frac{ds}{\vabs{s+U}^{\frac{7}3}}
	+ \frac{4 C_{2,3}^0(\kappa)}{81} \int_{-\infty}^0 \frac{ds}{\vabs{s+U}^{\frac{13}3}}
	\\
	\leq &\frac{2}{9\vabs{U}^{\frac43}} +
	\frac{28}{81 \vabs{U}^{\frac43}} \int_{-\infty}^0 \frac{ds}{\paren{1+s^2}^{\frac76}}
	+
	\frac{2 C_{2,3}^0(\kappa)}{9 \vabs{U}^{\frac73}} \int_{-\infty}^0 \frac{ds}{\paren{1+s^2}^{\frac{5}3}}
	\\
	&+ \frac{4 C_{2,3}^0(\kappa)}{81 \vabs{U}^{\frac{10}3}}
	\int_{-\infty}^0 \frac{ds}{\paren{1+s^2}^{\frac{13}6}} \\
	= &\paren{\frac29+\frac{14\sqrt{\pi}\,\Gamma(\frac23)}{81\Gamma(\frac76)}} \frac1{\vabs{U}^{\frac43}}
	+
	\frac{\sqrt{\pi} \, \Gamma(\frac76) C_{2,3}^0(\kappa)}{9\Gamma(\frac53)} \frac1{\vabs{U}^{\frac73}}
	+ 
	\frac{2\sqrt{\pi}\, \Gamma(\frac53) C_{2,3}^0(\kappa)}{81 \Gamma(\frac{13}6)}  \frac1{\vabs{U}^{\frac{10}3}}.
\end{split}
\]
Then
\[
\begin{split}
	\normInn{\FF_2[0]}_{\frac43} &\leq
	\frac29+
	\frac{14\sqrt{\pi}\,\Gamma(\frac23)}{81\, \Gamma(\frac76)} +
	\frac{\sqrt{\pi} \, \Gamma(\frac76) C_{2,3}^0(\kappa)}{9\,\Gamma(\frac53)} \frac1{\kappa} +
	\frac{2\sqrt{\pi}\, \Gamma(\frac53) C_{2,3}^0(\kappa)}{81 \Gamma(\frac{13}6)}  \frac1{\kappa^2} .
\end{split}
\]
An analogous result holds for $\FF_3[0]$.
\end{proof}


\subsection{Estimates for the derivatives of \texorpdfstring{$\RRR$}{R}}
\label{subsection:LipschitzRRR}

Let $\varrho_1,\varrho_2>1$ and denote
\[
	\alpha(\kappa,\varrho_1) =  \alpha_0(\kappa) \varrho_1, 
	\qquad
	\beta(\kappa,\varrho_2) = \beta_0(\kappa) \varrho_2,
\]
where functions $\alpha_0$ and $\beta_0$ as given in Proposition~\ref{proposition:firstIteration}. 
Notice that, for $\kappa\geq 1$,  the functions $\alpha(\kappa,\varrho_1)$ and $\beta(\kappa,\varrho_2)$ are positive functions decreasing in $\kappa$ and increasing in $\varrho_1$ and $\varrho_2$, respectively.

Then, we consider the closed set defined by
\[
\rectangle= \{ (\Wu,\Xu,\Yu) \in \XcalUTotal:
\normInn{\Wu}_{\frac83} \leq \alpha(\kappa,\varrho_1),  %
\normInn{\Xu}_{\frac43}, \normInn{\Yu}_{\frac43} \leq  \beta(\kappa,\varrho_2) \},
\]
where $\XcalUTotal$ and $\normInnTotal{\cdot}$ were defined in~\eqref{def:XcalUTotal} and~\eqref{def:norminnTotal}, respectively. 
The next two lemmas give estimates for the functions $\JJ$ and $\KK$, introduced in \eqref{def:hFunction} and \eqref{def:hamiltonianK} respectively,  and their derivatives for functions in the closed domain $\rectangle$.

In the following, we omit the dependence of certain functions in $\varrho_1,\varrho_2$ and $\kappa$ to simplify notation.

\begin{lemma}\label{lemma:LipschitzJJ}
Let $\kappa\geq 1$ and $\varrho_1, \varrho_2 > 1$  and define the functions 
\[
\begin{split}
\xi_0 &= 
\frac{16 + 216\beta}{81}
+ \frac{16\beta}{27\kappa}
+ \frac{16\alpha+48\beta^2}{27\kappa^2}
+ \frac{8\alpha\beta}{9\kappa^3}
+ \frac{4 \al^2 }{9 \kappa^4},
\\
\xi_1 &=
\frac{32+ 144\beta}{81}
+\frac{80 \beta}{81 \kappa}
+ \frac{64\al+192\beta^2}{81 \kappa^2}
+\frac{8\al\beta}{9 \kappa^3} 
+ \frac{8 \al^{2}}{27 \kappa^4},
\\
\xi_2 &=
\frac{16}{27}
+\frac{8 \beta}{9 \kappa}
+\frac{8 \al}{9 \kappa^2},
\\
\xi_3 &=
\frac43
+ \frac{8}{27 \kappa}
+\frac{16 \beta}{9 \kappa^2}
+ \frac{4\al}{9\kappa^3},
\\
\xi_4 &=
\frac{64}{81}
+\frac{8\beta}{9 \kappa}
+\frac{16 \alpha}{27 \kappa^2}, 
\\
\xi_5 &=
\frac89
+\frac{40}{81 \kappa}
+\frac{64 \beta}{27 \kappa^2}
+\frac{4 \alpha}{9 \kappa^3},
\end{split}
\]
which are positive functions decreasing in $\kappa$ and increasing in $\varrho_1$ and $\varrho_2$.

Then, for $\Zu \in \rectangle$, the following estimate for $\JJ(U,\Zu)$ in \eqref{def:hFunction} holds,
\[
\normInn{\JJ(\cdot,\Zu)}_2 
\leq \xi_0.
\]
Moreover, its first and second derivatives satisfy
\[
\begin{split}
\normInn{\partial_U \JJ(\cdot,\Zu)}_3
\leq \xi_1,
\quad
\normInn{\partial_W \JJ(\cdot,\Zu)}_{\frac43}
\leq \xi_2,
\quad
\normInn{\partial_X \JJ(\cdot,\Zu)}_{\frac23},
\normInn{\partial_Y \JJ(\cdot,\Zu)}_{\frac23},
\leq \xi_3,
\end{split}
\]
and 
  
\begin{align*}
&\normInn{\partial_{UW} \JJ(\cdot,\Zu)}_{\frac73}
\leq \xi_4,
&
&\normInn{\partial_{UX} \JJ(\cdot,\Zu)}_{\frac53},
\normInn{\partial_{UY} \JJ(\cdot,\Zu)}_{\frac53}
\leq \xi_5,
\\
&\normInn{\partial^2_{W} \JJ(\cdot,\Zu)}_{\frac23}
=\frac89,
&
&\normInn{\partial_{WX} \JJ(\cdot,\Zu)}_{1},
\normInn{\partial_{WY} \JJ(\cdot,\Zu)}_{1}
= \frac49,
\\
&\normInn{\partial_{XY} \JJ(\cdot,\Zu)}_{\frac43}
=\frac{10}9,
&
&\normInn{\partial^2_{X} \JJ(\cdot,\Zu)}_{\frac43},
\normInn{\partial^2_{Y} \JJ(\cdot,\Zu)}_{\frac43}
= \frac23.
\end{align*}
  
\end{lemma}
\begin{proof}
The estimate for $\JJ$ is a consequence of the following,
\[
\begin{split}
\vabs{\JJ(U,\Zu)} \leq& \,
\frac{4 \al^2 }{9 \vabs{U}^6}
+ \frac{16 \al}{27 \vabs{U}^4}
+ \frac{16}{81\vabs{U}^2}
+ \frac{8 \alpha \beta}{9\vabs{U}^{5}}
+ \frac{16\beta}{27 \vabs{U}^{3}}
+ \frac{8\beta}{3\vabs{U}^2}
+ \frac{2 \beta^2}{3 \vabs{U}^4}
+ \frac{10 \beta^2}{9 \vabs{U}^4}
\\
=& \,  \frac1{\vabs{U}^2}
\paren{
\frac{16 + 216\beta}{81}
+\frac{16\beta}{27\kappa}
+ \frac{16\alpha+{48}\beta^2}{27\kappa^2}
+ \frac{8\alpha\beta}{9\kappa^3}
+ \frac{4 \al^2 }{9 \kappa^4} }
:= \,
\frac{\xi_0}{\vabs{U}^2}.
\end{split}
\]
The first derivatives can be bounded by
\[
\begin{split}
\vabs{\partial_U \JJ(U,\Zu)} &\leq
  \frac{8 \al^{2}}{27 \vabs{U}^7} 
+ \frac{64 \al}{81 \vabs{U}^5}
+\frac{32}{81 \vabs{U}^{3}}
+\frac{8\al\beta}{9 \vabs{U}^{6}} 
+\frac{80 \beta}{81 \vabs{U}^4}
+\frac{16\beta}{9 \vabs{U}^3}
+\frac{8 \beta^2}{9 \vabs{U}^5}
+\frac{40 \beta^2}{27 \vabs{U}^5}
\\
&=
\frac1{\vabs{U}^3}\paren{
\frac{32+ 144\beta}{81}
+\frac{80 \beta}{81 \kappa}
+ \frac{64\al+192\beta^2}{81 \kappa^2}
+\frac{8\al\beta}{9 \kappa^3} 
+ \frac{8 \al^{2}}{27 \kappa^4} 
} := \frac{\xi_1}{\vabs{U}^3},
\\
\vabs{\partial_W \JJ(U,\Zu)} &\leq 
\frac{8 \al}{9 \vabs{U}^{\frac{10}{3}}}
+\frac{16}{27 \vabs{U}^{\frac{4}{3}}}
+\frac{8 \beta}{9 \vabs{U}^{\frac73}}
=
\frac1{\vabs{U}^{\frac43}}\paren{
\frac{16}{27}
+\frac{8 \beta}{9 \kappa}
+\frac{8 \al}{9 \kappa^2}
} := \frac{\xi_2}{\vabs{U}^{\frac43}},
\\
\vabs{\partial_X \JJ(U,\Zu)} &\leq
\frac{4\al}{9\vabs{U}^{\frac{11}3}} 
+ \frac{8}{27 \vabs{U}^{\frac53}}
+ \frac{4}{3 \vabs{U}^{\frac23}}
+\frac{2 \beta}{3 \vabs{U}^{\frac83}}
+\frac{10 \beta}{9 \vabs{U}^{\frac83}}
\\ &=
\frac1{\vabs{U}^{\frac23}} \paren{
\frac43
+ \frac{8}{27 \kappa}
+\frac{16 \beta}{9 \kappa^2}
+ \frac{4\al}{9\kappa^3} 
}
:= \frac{\xi_3}{\vabs{U}^{\frac23}},
\\
\vabs{\partial_Y \JJ(U,\Zu)} &\leq
\frac{\xi_3}{\vabs{U}^{\frac23}},
\end{split}
\]
and the second derivatives by
\[
\begin{split}
\vabs{\partial_{UW} \JJ(U,\Zu)} &\leq 
\frac{16 \alpha}{27 \vabs{U}^{\frac{13}3}}
+\frac{64}{81 \vabs{U}^{\frac73}}
+\frac{8\beta}{9 \vabs{U}^{\frac{10}3}}
=
\frac1{\vabs{U}^{\frac73}} \paren{
\frac{64}{81}
+\frac{8\beta}{9 \kappa}
+\frac{16 \alpha}{27 \kappa^2}}
:= \frac{\xi_4}{\vabs{U}^{\frac73}} ,
\\
\vabs{\partial_{UX} \JJ(U,\Zu)} &\leq
\frac{4 \alpha}{9 \vabs{U}^{\frac{14}3}}
+\frac{40 }{81 \vabs{U}^{\frac83}}
+\frac{8 }{9 \vabs{U}^{\frac53}}
+\frac{8 \beta}{9 \vabs{U}^{\frac{11}3}}
+\frac{40 \beta}{27 \vabs{U}^{\frac{11}3}}
\\ &=
\frac1{\vabs{U}^{\frac53}} \paren{
\frac89
+\frac{40}{81 \kappa}
+\frac{64 \beta}{27 \kappa^2}
+\frac{4 \alpha}{9 \kappa^3}
} 
:= \frac{\xi_5}{\vabs{U}^{\frac53}}, 
\\
\vabs{\partial_{UY} \JJ(U,\Zu)} &\leq 
\frac{\xi_5}{\vabs{U}^{\frac53}},
\end{split}
\]

\begin{align*}
\vabs{\partial_{W}^2 \JJ(U,\Zu)} &=
\frac{8}{9 \vabs{U}^{\frac23}}, &
\quad
\vabs{\partial_{WX} \JJ(U,\Zu)} &= 
\frac{4}{9 \vabs{U}}, 
&\quad
\vabs{\partial_{WY} \JJ(U,\Zu)} &= 
\frac{4}{9 \vabs{U}}, 
\\
\vabs{\partial_{X}^2 \JJ(U,\Zu)} &=
\frac{2}{3 \vabs{U}^{\frac43}}, 
&\quad
\vabs{\partial_{XY} \JJ(U,\Zu)} &=
\frac{10}{9 \vabs{U}^{\frac{4}{3}}}, 
&
\vabs{\partial_{Y}^2 \JJ(U,\Zu)} &=
\frac{2}{3 \vabs{U}^{\frac{4}{3}}}. 
\end{align*}
  
Finally notice that, $\kappa\geq 1$ and $\varrho_1, \varrho_2 > 1$,  by Proposition~\ref{proposition:firstIteration}$, \alpha(\kappa,\varrho_1)=\varrho_1\alpha_0(\kappa)$ and $\beta(\kappa,\varrho_2)=\varrho_2\beta_0(\kappa)$ are positive functions decreasing in $\kappa$ and increasing in $\varrho_1$ and $\varrho_2$.
Therefore, the auxiliary functions $\xi_0,..,\xi_5$ are as well.
\end{proof}

\begin{lemma}
Let	$\varrho_1 \in (1,60)$, $\varrho_2 \in (1,3)$ and $\kappa\geq3$. 
Then, for $\Zu \in \rectangle$, the derivatives of the Hamiltonian $\KK$ in \eqref{def:hamiltonianK} satisfy
\[
\begin{split}
\normInn{\partial_U \KK(\cdot,\Zu)}_{\frac{11}3} \leq \eta_1,
\quad
\normInn{\partial_W \KK(\cdot,\Zu)}_2 \leq \eta_2,
\quad
\normInn{\partial_X \KK(\cdot,\Zu)}_{\frac43}, 
\normInn{\partial_Y \KK(\cdot,\Zu)}_{\frac43} \leq \eta_3 
\end{split}
\]
and 
  
\begin{align*}
\normInn{\partial_{UW}\KK(\cdot,\Zu)}_3 &\leq \eta_4, 
&
\normInn{\partial_{UX}\KK(\cdot,\Zu)}_{\frac73},
\normInn{\partial_{UY}\KK(\cdot,\Zu)}_{\frac73} &\leq \eta_5,
\\
\normInn{\partial^2_{W}\KK(\cdot,\Zu)}_{-\frac23} &\leq \eta_6, 
&
\normInn{\partial_{WX}\KK(\cdot,\Zu)}_{\frac53},
\normInn{\partial_{WY}\KK(\cdot,\Zu)}_{\frac53} &\leq \eta_7,
\\
\normInn{\partial_{XY}\KK(\cdot,\Zu)}_{2} &\leq \eta_8, 
&
\normInn{\partial^2_{X}\KK(\cdot,\Zu)}_{2},
\normInn{\partial^2_{Y}\KK(\cdot,\Zu)}_{2} &\leq \eta_9,
\end{align*}
  
where
  
\begin{align*}
	\eta_0 &= \sqrt{1 - \frac{\xi_0}{\kappa^2}}, &
	\eta_1 &= \frac{4 \xi_0}{9 \eta_0 (4-\xi_0\eta_0^{-1}\kappa^{-2})}
	+ \frac{\xi_1}{6 \eta_0^3}
	+ \frac{\alpha^2}{2 \kappa^2}, \\
	\eta_2 &= \frac{3\al}{2} + \frac{\xi_2}{6 \eta_0^3}, &
	\eta_3 &= \frac{\xi_3}{ 6\eta_0^3}, \\
	\eta_4 &= \al
	+\frac{\xi_2}{9 \eta_0^3}
	+\frac{\xi_4}{6 \eta_0^3}
	+\frac{\xi_1 \xi_2}{4 \eta_0^5 \kappa^2}, &
	\eta_5 &= 
	\frac{\xi_3}{9 \eta_0^3} 
	+\frac{\xi_5}{6 \eta_0^3} 
	+\frac{\xi_1\xi_3}{4 \eta_0^5 \kappa^2}, \\
	\eta_6 &= \frac32
	+ \frac{4}{27 \eta_0^3 \kappa^2}
	+ \frac{\xi_2^2}{4 \eta_0^5 \kappa^4}, &
	\eta_7 &= \frac{2}{27 \eta_0^3}
	+\frac{\xi_2 \xi_3}{4 \eta_0^5 \kappa}, \\
	\eta_8 &= \frac{5}{27 \eta_0^3}
	+ \frac{\xi_3^2}{4 \eta_0^5} , &
	\eta_9 &= \frac{1}{9 \eta_0^3}
	+ \frac{\xi_3^2}{4 \eta_0^5}.
\end{align*}
  
Moreover, for $\kappa\geq 3$, $\varrho_1 \in (1,60)$ and $\varrho_2 \in (1,3)$, the  functions $\eta_1,..,\eta_9$ are positive functions decreasing in $\kappa$ and increasing in $\varrho_1$ and $\varrho_2$.
\end{lemma}

\begin{proof}
Let us consider first lower bounds for $1+\JJ(U,\Zu)$. 
By Lemma~\ref{lemma:LipschitzJJ}, one has that
\[
\vabs{\JJ(U,\Zu)} \leq \frac{\xi_0}{\kappa^2}.
\]		
Notice that,  by Proposition~\ref{proposition:firstIteration}, $\alpha_0(\kappa)$ and $\beta_0(\kappa)$ are decreasing for positive values of $\kappa$.
Therefore, $\alpha(\kappa,\varrho_1)$ and $\beta(\kappa,\varrho_2)$ are decreasing for $\kappa>0$ and increasing for $\varrho_1, \varrho_2>0$.
Then, by the definition of $\xi_0$ in Lemma~\ref{lemma:LipschitzJJ},
\begin{equation}\label{proof:estimateJJLipschitz}
\vabs{\JJ(U,\Zu)} \leq \frac{\xi_0}{\kappa^2}\Bigg|_{\kappa=3,\varrho_1=60, \varrho_2=3} < 1.
\end{equation}
Therefore, by the triangular inequality,
\begin{equation}\label{proof:LipzchitzDenominator1}
\vabss{\sqrt{1+\JJ(U,\Zu)}} \geq \sqrt{1 - \vabs{\JJ(U,\Zu)}} \geq \sqrt{1 - \frac{\xi_0}{\kappa^2}}:= \eta_0.
\end{equation}

Next, we consider lower bounds for the denominator $1+\sqrt{1+\JJ(U,\Zu)}$. 
Notice that, by the mean value theorem and taking into account~\eqref{proof:estimateJJLipschitz},
\[
\begin{split}
	\vabs{\sqrt{1+\JJ(U,\Zu)} - 1} &\leq \vabs{\JJ(U,\Zu)} \sup_{\sigma \in [0,1]}
	\vabs{\frac1{2\sqrt{1+\sigma\JJ(U,\Zu)}}}
	\\
	&\leq \frac{\xi_0}{2\kappa^2} \frac1{\sqrt{1-\vabs{\JJ(U,\Zu)}}}
	\leq \frac{\xi_0}{2 \eta_0 \kappa^2}.
\end{split}
\]
In addition, one can see that,
\[
\vabs{\sqrt{1+\JJ(U,\Zu)} - 1} 
\leq 
\frac{\xi_0}{2 \eta_0 \kappa^2}\Bigg|_{\kappa=3,\varrho_1=60, \varrho_2=3} < 2.
\]
Then, by the triangular inequality,
\begin{equation} \label{proof:LipzchitzDenominator2}
	\vabs{1 + \sqrt{1+\JJ(U,\Zu)}} \geq 2- \frac{\xi_0}{2\eta_0\kappa^2}.
\end{equation}

We estimate now the first derivatives of $\KK$. 
By the formulae in Appendix~\ref{appendix:formulasRRR}, Lemma~\ref{lemma:LipschitzJJ}, \eqref{proof:LipzchitzDenominator1} and \eqref{proof:LipzchitzDenominator2}, one has that
\[
\begin{split}
\vabs{\partial_U \KK(U,\Zu)} &\leq
\frac{\alpha^2}{2 \vabs{U}^{\frac{17}3}}
+ \frac{2 \xi_0}{9 \eta_0 (2-\xi_0 2^{-1}\eta_0^{-1}\kappa^{-2}) \vabs{U}^{\frac{11}3}}
+\frac{\xi_1}{6  \eta_0^3 \vabs{U}^{\frac{11}3}}
\\
&\leq \frac1{\vabs{U}^{\frac{11}3}} \paren{
	\frac{4 \xi_0}{9 \eta_0 (4-\xi_0\eta_0^{-1}\kappa^{-2})}
	+ \frac{\xi_1}{6 \eta_0^3}
	+ \frac{\alpha^2}{2 \kappa^2}
}
:= \frac{\eta_1}{\vabs{U}^{\frac{11}3}}
\end{split}
\]
and
\[
\begin{split}
	\vabs{\partial_W \KK(U,\Zu)} &=
	\frac{3\al}{2\vabs{U}^{2}}
	+ \frac{\xi_2}{6 \eta_0^3 \vabs{U}^{2}}
	:= \frac{\eta_2}{\vabs{U}^2},
	\\
	\vabs{\partial_X \KK(U,\Zu)} &=
	\frac{\xi_3}{ \vabs{U}^{\frac43} 6\eta_0^3} 
	:= \frac{\eta_3}{\vabs{U}^{\frac43}},
	\\
	\vabs{\partial_Y \KK(U,\Zu)} &\leq \frac{\eta_3}{\vabs{U}^{\frac43}}.
\end{split}
\]

Finally, we consider estimates for the second derivatives of $\KK$,
\[
\begin{split}
\vabs{\partial_{UW} \KK(U,\Zu)} &\leq
\frac{\al}{\vabs{U}^{3}}
+\frac{\xi_2}{9 \eta_0^3 \vabs{U}^{3}}
+\frac{\xi_4}{6 \eta_0^3 \vabs{U}^{3}}
+\frac{\xi_1 \xi_2}{4 \eta_0^5 \vabs{U}^{5}}
\\
&\leq
\frac1{\vabs{U}^3} \paren{
\al
+\frac{\xi_2}{9 \eta_0^3}
+\frac{\xi_4}{6 \eta_0^3}
+\frac{\xi_1 \xi_2}{4 \eta_0^5 \kappa^2}
}
:= \frac{\eta_4}{\vabs{U}^3},
\\
\vabs{\partial_{UX} \KK(U,\Zu)} &\leq
\frac{\xi_3}{9 \eta_0^3 \vabs{U}^{\frac73}} 
+\frac{\xi_5}{6 \eta_0^3 \vabs{U}^{\frac73}} 
+\frac{\xi_1\xi_3}{4 \eta_0^5 \vabs{U}^{\frac{13}3}}
\\
&\leq
\frac1{\vabs{U}^{\frac73}} \paren{
\frac{\xi_3}{9 \eta_0^3} 
+\frac{\xi_5}{6 \eta_0^3} 
+\frac{\xi_1\xi_3}{4 \eta_0^5 \kappa^2}
}:= \frac{\eta_5}{\vabs{U}^{\frac73}},
\\
\vabs{\partial_{UY} \KK(U,\Zu)} &\leq
\frac{\eta_5}{\vabs{U}^{\frac73}}
\end{split}
\]
and
\[
\begin{split}
\vabs{\partial_W^2 \KK(U,\Zu)} 
&\leq
\frac{3 \vabs{U}^{\frac23}}{2}
+ \frac{4}{27 \eta_0^3 \vabs{U}^{\frac43}}
+ \frac{\xi_2^2}{4 \eta_0^5 \vabs{U}^{\frac{10}3}}
\\
&\leq
\vabs{U}^{\frac23} \paren{\frac32
+ \frac{4}{27 \eta_0^3 \kappa^2}
+ \frac{\xi_2^2}{4 \eta_0^5 \kappa^4}
}
:= \vabs{U}^{\frac23} \eta_6,
\\
\vabs{\partial_{WX} \KK(U,\Zu)} &\leq
\frac{2}{27 \eta_0^3 \vabs{U}^{\frac53}}
+\frac{\xi_2 \xi_3}{4 \eta_0^5 \vabs{U}^{\frac83}}
\leq \frac1{\vabs{U}^{\frac53}} \paren{
\frac{2}{27 \eta_0^3}
+\frac{\xi_2 \xi_3}{4 \eta_0^5 \kappa}
}
:= \frac{\eta_7}{\vabs{U}^{\frac53}},
\\
\vabs{\partial_{WY} \KK(U,\Zu)} &\leq
\frac{\eta_7}{\vabs{U}^{\frac53}}
\end{split}
\]
and
\[
\begin{split}
	\vabs{\partial_{XY} \KK(U,\Zu)} &\leq
	\frac{5}{27 \eta_0^3 \vabs{U}^{2}}
	+ \frac{\xi_3^2}{4 \eta_0^5 \vabs{U}^{2}}
	:= \frac{\eta_8}{\vabs{U}^2},
	\\
	\vabs{\partial^2_{X} \KK(U,\Zu)} &\leq
	\frac{1}{9 \eta_0^3 \vabs{U}^2}
	+ \frac{\xi_3^2}{4 \eta_0^5 \vabs{U}^{2}}
	:= 
	\frac{\eta_9}{\vabs{U}^2},
	\\
	\vabs{\partial^2_{Y} \KK(U,\Zu)} &\leq
	\frac{\eta_9}{\vabs{U}^2}.
\end{split}
\]
Since $\xi_0,\cdots ,\xi_5$ are decreasing functions for $\kappa$ and increasing for $\varrho$, by construction one can see that the  functions $\eta_1,..,\eta_9$ are increasing for 	$\varrho_1 \in (1,60)$ and $\varrho_2 \in (1,3)$ and decreasing for $\kappa\geq 3$.
\end{proof}
The two previous lemmas provide the necessary bounds to estimate the derivatives of the function $\RRR$.
\begin{lemma} \label{lemma:estimatesDerivativesRRRStokes}
Assume 	$\varrho_1 \in (1,60)$, $\varrho_2 \in (1,3)$ and $\kappa\geq3$. Then, for $\Zu \in \rectangle$, one has that
  
\begin{align*}
\normInn{\partial_W \RRR_1[\Zu]}_3 &\leq \nu_1,
&
\normInn{\partial_X \RRR_1[\Zu]}_{\frac73} &\leq \nu_2, 
&
\normInn{\partial_Y \RRR_1[\Zu]}_{\frac73}
&\leq \nu_2,
\\
\normInn{\partial_W \RRR_2[\Zu]}_{\frac23} &\leq \nu_3,
&
\normInn{\partial_X \RRR_2[\Zu]}_{2} &\leq \nu_4, 
&
\normInn{\partial_Y \RRR_2[\Zu]}_{2} &\leq \nu_5,
\\
\normInn{\partial_W \RRR_3[\Zu]}_{\frac23} &\leq \nu_3,
&
\normInn{\partial_X \RRR_3[\Zu]}_{2} &\leq  \nu_5,
&
\normInn{\partial_Y \RRR_3[\Zu]}_{2} &\leq \nu_4,
\end{align*}
  
where
\[
\begin{split}
	\nu_0 &= \paren{1 - \frac{\eta_2}{\kappa^2}}^2, 
	\\
	\nu_1 &= \frac{\eta_4}{\nu_0}
	+ \frac{\eta_2\eta_4}{\nu_0 \kappa^2}
	+ \frac{\eta_1\eta_6}{\nu_0},
	\\
	\nu_2 &= \frac{\eta_5}{\nu_0}
	+
	\frac{\eta_2\eta_5}{\nu_0\kappa^{2}}
	+
	\frac{\eta_1\eta_7}{\nu_0\kappa^3},
\end{split}
\]
\[
\begin{split}
	\nu_3 &= \frac{\eta_7}{\nu_0\kappa} +
	\frac{\eta_2\eta_7}{\nu_0\kappa^3} +
	\frac{\beta \eta_6}{\nu_0} +
	\frac{2\beta \eta_2 \eta_6}{\nu_0 \kappa^2} +
	\frac{\eta_3\eta_6}{\nu_0},
 \\
	\nu_4 &= 
	\frac{\eta_8}{\nu_0} +  \frac{\eta_2\eta_8}{\nu_0\kappa^2} +
	\frac{\eta_2}{\nu_0} + \frac{\eta_2^2}{\nu_0 \kappa^2} + 
	\frac{\beta \eta_7}{\nu_0 \kappa} + \frac{2\beta \eta_2 \eta_7}{\nu_0 \kappa^3} +
	\frac{\eta_3\eta_7}{\nu_0 \kappa},
	\\
	\nu_5 &=
	\frac{\eta_9}{\nu_0} + \frac{\eta_2\eta_9}{\nu_0\kappa^2} +
	\frac{\beta\eta_7}{\nu_0\kappa} + \frac{2\beta\eta_2\eta_7}{\nu_0\kappa^3} +
	\frac{\eta_3\eta_7}{\nu_0\kappa}.
\end{split}
\]
Moreover, for $\varrho_1 \in (1,60)$, $\varrho_2 \in (1,3)$ and $\kappa\geq 3$, the  functions $\nu_1,..,\nu_5$ are positive functions decreasing in $\kappa$ and increasing in $\varrho_1$ and $\varrho_2$.
\end{lemma}

\begin{proof}
Let us first look for a lower bound for the denominator $1+\partial_W\KK(U,\Zu)$.
Notice that,
\[
\begin{split}
\vabs{\partial_W \KK(U,\Zu)} \leq \frac{\eta_2}{\kappa^2}\Bigg|_{\kappa=3,\varrho_1=60, \varrho_2=3}
< 1.
\end{split}
\]
Then, by the triangular inequality,
\[
\vabs{1 + \partial_W \KK(U,\Zu)}^2 \geq 
\paren{1 - \frac{\eta_2}{\kappa^2}}^2:= \nu_0.
\]
We now analyze the derivatives of $\RRR_1[\Zu]$. Indeed, by the formulae in Appendix~\ref{appendix:formulasRRR}, one has that
\[
\begin{split}
\vabs{\partial_W \RRR_1[\Zu](U)} 
&\leq
\frac{\eta_4}{\nu_0 \vabs{U}^3}
+ \frac{\eta_2\eta_4}{\nu_0 \vabs{U}^5}
+ \frac{\eta_1\eta_6}{\nu_0 \vabs{U}^3}
\\
&\leq
\frac1{\vabs{U}^3} \paren{
\frac{\eta_4}{\nu_0}
+ \frac{\eta_2\eta_4}{\nu_0 \kappa^2}
+ \frac{\eta_1\eta_6}{\nu_0}
}
:= \frac{\nu_1}{\vabs{U}^3},
\\
\vabs{\partial_X \RRR_1[\Zu](U)} &\leq
\frac{\eta_5}{\nu_0\vabs{U}^\frac73}
+
\frac{\eta_2\eta_5}{\nu_0\vabs{U}^{\frac{13}3}}
+
\frac{\eta_1\eta_7}{\nu_0\vabs{U}^{\frac{16}3}}
\\
&\leq
\frac1{\vabs{U}^{\frac73}} \paren{
\frac{\eta_5}{\nu_0}
+
\frac{\eta_2\eta_5}{\nu_0\kappa^{2}}
+
\frac{\eta_1\eta_7}{\nu_0\kappa^3}
}
:= \frac{\nu_2}{\vabs{U}^{\frac73}},
\\
\vabs{\partial_Y \RRR_1[\Zu](U)} &\leq
\frac{\nu_2}{\vabs{U}^{\frac73}}.
\end{split}
\]
The derivatives of $\RRR_2[\Zu]$ satisfy
\[
\begin{split}
\vabs{\partial_W \RRR_2[\Zu](U)} 
&\leq
\frac1{\nu_0}
\paren{\frac{\eta_7}{\vabs{U}^{\frac53}}
+
\frac{\beta \eta_6}{\vabs{U}^{\frac23}}}
\paren{1 + \frac{\eta_2}{\vabs{U}^2}}
+ \frac1{\nu_0}\paren{
\frac{\eta_3}{\vabs{U}^{\frac43}}
+ 
\frac{\beta\eta_2}{\vabs{U}^{\frac{10}3}}}
\eta_6 \vabs{U}^{\frac23}
\\
&\leq 
\frac1{\vabs{U}^{\frac23}} \paren{
\frac{\eta_7}{\nu_0\kappa} +
\frac{\eta_2\eta_7}{\nu_0\kappa^3} +
\frac{\beta \eta_6}{\nu_0} +
\frac{2\beta \eta_2 \eta_6}{\nu_0 \kappa^2} +
\frac{\eta_3\eta_6}{\nu_0} 
}
:= \frac{\nu_3}{\vabs{U}^{\frac23}}, 
\\
\vabs{\partial_X \RRR_2[\Zu](U)} 
&\leq 
\frac1{\nu_0}
\paren{
\frac{\eta_8}{\vabs{U}^2} +
\frac{\eta_2}{\vabs{U}^2} +
\frac{\beta \eta_7}{\vabs{U}^3}
} \paren{
1+\frac{\eta_2}{\vabs{U}^2}
}
+ \frac1{\nu_0} \paren{
\frac{\eta_3}{\vabs{U}^{\frac43}}
+ \frac{\beta \eta_2}{\vabs{U}^{\frac{10}3}}
} \frac{\eta_7}{\vabs{U}^{\frac53}}
\\
&\leq \frac1{\vabs{U}^2} \paren{
\frac{\eta_8}{\nu_0} +  \frac{\eta_2\eta_8}{\nu_0\kappa^2} +
\frac{\eta_2}{\nu_0} + \frac{\eta_2^2}{\nu_0 \kappa^2} + 
\frac{\beta \eta_7}{\nu_0 \kappa} + \frac{2\beta \eta_2 \eta_7}{\nu_0 \kappa^3} +
\frac{\eta_3\eta_7}{\nu_0 \kappa}
} := \frac{\nu_4}{\vabs{U}^2},
\end{split}
\]
\[
\begin{split}
\vabs{\partial_Y \RRR_2[\Zu](U)} 
&\leq 
\frac1{\nu_0} \paren{
\frac{\eta_9}{\vabs{U}^2} +
\frac{\beta\eta_7}{\vabs{U}^3}
} \paren{
1 + \frac{\eta_2}{\vabs{U}^2}} + 
\frac1{\nu_0}\paren{
\frac{\eta_3}{\vabs{U}^{\frac43}} +
\frac{\beta\eta_2}{\vabs{U}^{\frac{10}3}} }
\frac{\eta_7}{\vabs{U}^{\frac53}}
\\
&\leq \frac1{\vabs{U}^2} \paren{
\frac{\eta_9}{\nu_0} + \frac{\eta_2\eta_9}{\nu_0\kappa^2} +
\frac{\beta\eta_7}{\nu_0\kappa} + \frac{2\beta\eta_2\eta_7}{\nu_0\kappa^3} +
\frac{\eta_3\eta_7}{\nu_0\kappa} 
} := \frac{\nu_5}{\vabs{U}^2}.
\end{split}
\]
An analogous procedure leading to the same estimates holds for the derivatives of $\RRR_3$.

Finally, since $\eta_1,\cdots ,\eta_9$ are decreasing functions for $\kappa$ and increasing for $\varrho_1$ and $\varrho_2$,  one can see that the  functions $\nu_1,\cdots ,\nu_5$ are increasing for $\varrho_1 \in (1,60)$ and $\varrho_2 \in (1,3)$ and decreasing for $\kappa\geq 3$.
\end{proof}

\subsection{The Lipschitz constant of \texorpdfstring{$\FF$}{F}}
\label{subsection:operatorFF}
The next lemma gives estimates for the Lipschitz constant of the operator $\FF$ in \eqref{def:operatorFF} with respect to each variable.
\begin{proposition} \label{proposition:LipschitzStokes}
Assume $\gamma\in \left(0,\arctan\frac{\sqrt{3}}2\right)$, $\kappa \geq 3$,	$\varrho_1 \in (1,60)$, $\varrho_2 \in (1,3)$.
%
Then, for any $\Zu$, $\ZuHat \in \rectangle$, one has that
\[
\begin{split}
	\normInn{{\FF_1}[\Zu]-{\FF_1}[\ZuHat]}_{\frac{8}{3}} &\leq 
	\frac{\wt{\nu}_1}{\kInn^2} \normInn{\Wu - \WuHat}_{\frac{8}{3}} 
	+ \wt{\nu}_2 \normInn{\Xu - \XuHat}_{\frac{4}{3}}
	+ \wt{\nu}_2 \normInn{\Yu - \YuHat}_{\frac{4}{3}},
	\\
	\normInn{{\FF_2}[\Zu]-{\FF_2}[\ZuHat]}_{\frac{4}{3}} &\leq 
	\frac{\wt{\nu}_3}{\kInn^2} \normInn{\Wu - \WuHat}_{\frac{8}{3}} 
	+ \frac{\wt{\nu}_4}{\kInn^2} \normInn{\Xu - \XuHat}_{\frac{4}{3}}
	+ \frac{\wt{\nu}_5}{\kInn^2} \normInn{\Yu - \YuHat}_{\frac{4}{3}},
	\\
	\normInn{{\FF_3}[\Zu]-{\FF_3}[\ZuHat]}_{\frac{4}{3}} &\leq 
	\frac{\wt{\nu}_3}{\kInn^2} \normInn{\Wu - \WuHat}_{\frac{8}{3}} 
	+ \frac{\wt{\nu}_5}{\kInn^2} \normInn{\Xu - \XuHat}_{\frac{4}{3}}
	+ \frac{\wt{\nu}_4}{\kInn^2} \normInn{\Yu - \YuHat}_{\frac{4}{3}},
\end{split}
\]
where
\[
	\wt{\nu}_i = \nu_i {\frac{\sqrt{\pi} \Gamma(\frac43)}{2\Gamma(\frac{11}{6}) }},
	\quad \text{for } i=1,2, 
    \qquad
	\wt{\nu}_j = \frac{\nu_j}{\sin\gamma (\cos\gamma)^{\frac43}},
	\quad \text{for } j=3,4,5
\]
and $\nu_j$ are the functions introduced in Lemma \ref{lemma:estimatesDerivativesRRRStokes}.
Moreover, for $\kappa\geq 3$, $\varrho_1 \in (1,60)$ and $\varrho_2 \in (1,3)$, the  functions $\wt{\nu}_i$, $i=1\ldots 5$,  are positive functions decreasing in $\kappa$ and increasing in $\varrho_1$ and $\varrho_2$.

This implies that 
\[
\normInnTotal{\FF[\Zu]-\FF[\ZuHat]} \leq L \normInnTotal{\Zu-\ZuHat},
\]
where
\begin{equation}\label{def:L}
L= \max\claus{\frac{\wt{\nu}_1}{\kappa^2}+2\wt{\nu}_2,
\frac{\wt{\nu}_3+\wt{\nu}_4+\wt{\nu}_5}{\kappa^2}}.
\end{equation}
\end{proposition}

\begin{proof}

To estimate the Lipschitz constant, we first estimate each component $\RRR_j[\Zu]-\RRR_j[\ZuHat]$ separately  for $j=1,2,3$. By the mean value theorem we have
\[
	\RRR_j[\Zu]-\RRR_j[\ZuHat] 
	=\boxClaus{\int_0^1 D \RRR_j[s \Zu + (1-s) \ZuHat]  d s }
	(\Zu - \ZuHat) .
\]
Then, for $j=2,3$,
\[
\begin{split}
	\MoveEqLeft[4]{\normInn{\RRR_1[\Zu]-\RRR_1[\ZuHat]}_{\frac{11}{3}} \leq
		\normInn{\Wu - \WuHat}_{\frac{8}{3}}
		\sup_{\varphi \in {B(\varrho)}}  \normInn{\partial_W\RRR_1[\varphi]}_{1} }\\
	&+ \normInn{\Xu - \XuHat}_{\frac{4}{3}}
	\sup_{\varphi \in {B(\varrho)}} \normInn{\partial_X{\RRR_1}[\varphi]}_{\frac{7}{3}} 
	+ \normInn{\Yu - \YuHat}_{\frac{4}{3}} \sup_{\varphi \in {B(\varrho)}} 
	\normInn{\partial_Y\RRR_1[\varphi]}_{\frac{7}{3}}, \\
	\MoveEqLeft[4]{\normInn{\RRR_j[\Zu]-\RRR_j[\ZuHat]}_{\frac{4}{3}} \leq
		\normInn{\Wu - \WuHat}_{\frac{8}{3}}
		\sup_{\varphi \in {B(\varrho)}} \normInn{\partial_W\RRR_j[\varphi]}_{-\frac{4}{3}} }\\
	&+ \normInn{\Xu - \XuHat}_{\frac{4}{3}} 
	\sup_{\varphi \in {B(\varrho)}} \normInn{\partial_X{\RRR_j}[\varphi]}_{0} 
	+ \normInn{\Yu - \YuHat}_{\frac{4}{3}}
	\sup_{\varphi \in {B(\varrho)} } 
	\normInn{\partial_Y \RRR_j[\varphi]}_{0}.
\end{split}
\]
Applying Lemma~\ref{lemma:estimatesDerivativesRRRStokes}, we obtain
\[
\begin{split}
	\normInn{{\RRR_1}[\Zu]-{\RRR_1}[\ZuHat]}_{\frac{11}{3}} &\leq 
	 \frac{\nu_1}{\kInn^2} \normInn{\Wu - \WuHat}_{\frac{8}{3}} 
	+ \nu_2 \normInn{\Xu - \XuHat}_{\frac{4}{3}}
	+ \nu_2 \normInn{\Yu - \YuHat}_{\frac{4}{3}},
	\\
	\normInn{{\RRR_2}[\Zu]-{\RRR_2}[\ZuHat]}_{\frac{4}{3}} &\leq 
	 \frac{\nu_3}{\kInn^2} \normInn{\Wu - \WuHat}_{\frac{8}{3}} 
	+ \frac{\nu_4}{\kInn^2} \normInn{\Xu - \XuHat}_{\frac{4}{3}}
	+ \frac{\nu_5}{\kInn^2} \normInn{\Yu - \YuHat}_{\frac{4}{3}},
	\\
	\normInn{{\RRR_3}[\Zu]-{\RRR_3}[\ZuHat]}_{\frac{4}{3}} &\leq 
	\frac{\nu_3}{\kInn^2} \normInn{\Wu - \WuHat}_{\frac{8}{3}} 
	+ \frac{\nu_5}{\kInn^2} \normInn{\Xu - \XuHat}_{\frac{4}{3}}
	+ \frac{\nu_4}{\kInn^2} \normInn{\Yu - \YuHat}_{\frac{4}{3}}.
\end{split}
\]
Finally, applying Lemma~\ref{lemma:boundsOperatorGG}, for all $0<\sigma\leq\gamma$, we obtain the following estimates 
\[
\begin{split}
	\normInn{{\FF_1}[\Zu]-{\FF_1}[\ZuHat]}_{\frac83} &\leq 
	\frac{\sqrt{\pi} \Gamma(\frac43)}{2\Gamma(\frac{11}{6}) }
 \boxClaus{
	\frac{\nu_1}{\kInn^2} \normInn{\Wu - \WuHat}_{\frac{8}{3}} 
	+ \nu_2 \normInn{\Xu - \XuHat}_{\frac{4}{3}}
	+ \nu_2 \normInn{\Yu - \YuHat}_{\frac{4}{3}} },
	\\
	\normInn{{\FF_2}[\Zu]-{\FF_2}[\ZuHat]}_{\frac{4}{3}} &\leq 
	\frac1{\sin\sigma (\cos\sigma)^{\frac43}}\boxClaus{
	\frac{\nu_3}{\kInn^2} \normInn{\Wu - \WuHat}_{\frac{8}{3}} 
	+ \frac{\nu_4}{\kInn^2} \normInn{\Xu - \XuHat}_{\frac{4}{3}}
	+ \frac{\nu_5}{\kInn^2} \normInn{\Yu - \YuHat}_{\frac{4}{3}}},
	\\
	\normInn{{\FF_3}[\Zu]-{\FF_3}[\ZuHat]}_{\frac{4}{3}} &\leq 
	\frac1{\sin\sigma (\cos\sigma)^{\frac43}}
	\boxClaus{
	\frac{\nu_3}{\kInn^2} \normInn{\Wu - \WuHat}_{\frac{8}{3}} 
	+ \frac{\nu_5}{\kInn^2} \normInn{\Xu - \XuHat}_{\frac{4}{3}}
	+ \frac{\nu_4}{\kInn^2} \normInn{\Yu - \YuHat}_{\frac{4}{3}}}.
\end{split}
\]
Let us notice that $f(\sigma)=\frac1{\sin\sigma (\cos\sigma)^{\frac43}}$ has a minimum at $\sigma_*=\arctan(\sqrt3/2)$ and is decreasing for $\sigma\in (0,\gamma)$ since we are assuming that $\gamma \in (0, \arctan(\sqrt3/2))$.

Then, defining
  
\begin{align*}
	\wt{\nu}_1 &= {\frac{\sqrt{\pi} \Gamma(\frac43)}{2\Gamma(\frac{11}{6}) }}\nu_1,
	&
	\wt{\nu}_2 &= {\frac{\sqrt{\pi} \Gamma(\frac43)}{2\Gamma(\frac{11}{6}) }} \nu_2, 
	\\
	\wt{\nu}_3 &= f(\gamma)\nu_3,
	&
	\wt{\nu}_4 &= f(\gamma)\nu_4,
	&
	\wt{\nu}_5 &= f(\gamma)\nu_5,
\end{align*}
  
one has that
\[
\begin{split}
	\normInn{{\FF_1}[\Zu]-{\FF_1}[\ZuHat]}_{\frac{8}{3}} &\leq 
	\frac{\wt{\nu}_1}{\kInn^2} \normInn{\Wu - \WuHat}_{\frac{8}{3}} 
	+ \wt{\nu}_2 \normInn{\Xu - \XuHat}_{\frac{4}{3}}
	+ \wt{\nu}_2 \normInn{\Yu - \YuHat}_{\frac{4}{3}},
	\\
	\normInn{{\FF_2}[\Zu]-{\FF_2}[\ZuHat]}_{\frac{4}{3}} &\leq 
	\frac{\wt{\nu}_3}{\kInn^2} \normInn{\Wu - \WuHat}_{\frac{8}{3}} 
	+ \frac{\wt{\nu}_4}{\kInn^2} \normInn{\Xu - \XuHat}_{\frac{4}{3}}
	+ \frac{\wt{\nu}_5}{\kInn^2} \normInn{\Yu - \YuHat}_{\frac{4}{3}},
	\\
	\normInn{{\FF_3}[\Zu]-{\FF_3}[\ZuHat]}_{\frac{4}{3}} &\leq 
	\frac{\wt{\nu}_3}{\kInn^2} \normInn{\Wu - \WuHat}_{\frac{8}{3}} 
	+ \frac{\wt{\nu}_5}{\kInn^2} \normInn{\Xu - \XuHat}_{\frac{4}{3}}
	+ \frac{\wt{\nu}_4}{\kInn^2} \normInn{\Yu - \YuHat}_{\frac{4}{3}}.
\end{split}
\]
   
\end{proof}


From now on, we emphasize that all the constants $\xi_j, \eta_j, \nu_j$ and $\wt \nu_j$ are in fact functions of $\varrho_1, \varrho_2$, $\kappa$ and $\gamma$. From now on we will write this dependence explicitly. 



To apply Proposition~\ref{proposition:LipschitzStokes} we first need to impose that $\FF : \rectangle \to \rectangle$ is well defined.

\begin{proposition}\label{proposition:Fwelldefined}
Assume 	$\varrho_1 \in (1,60)$, $\varrho_2 \in (1,3)$, $\kappa \geq 3$, $\gamma\in(0,\arctan(\frac{\sqrt3}2))$ and denote
\[
\begin{split}
	g_1(\kappa,\varrho_1,\varrho_2, \gamma) 
	=&
	\paren{\varrho_1 - 1 - \frac{\wt{\nu}_1(\kappa,\varrho_1,\varrho_2)}{\kappa^2}\varrho_1}\alpha_0(\kappa) -
	2 \wt{\nu}_2(\kappa,\varrho_1,\varrho_2) \varrho_2 \beta_0(\kappa),
	\\
	g_2(\kappa,\varrho_1,\varrho_2,\gamma) =&
	\paren{\varrho_2 - 1 - \frac{\wt{\nu}_4(\kappa,\varrho_1,\varrho_2,\gamma) + \wt{\nu}_5(\kappa,\varrho_1,\varrho_2,\gamma)}{\kappa^2}\varrho_2 }\beta_0(\kappa)
 \\
	&- \frac{\wt{\nu}_3(\kappa,\varrho_1,\varrho_2,\gamma)}{\kappa^2}\varrho_1\alpha_0(\kappa).
\end{split}
\]
Then, $\FF: \rectangle \to \rectangle$ is well-defined provided
\begin{equation}\label{def:g1g2condition}
	g_1(\kappa,\varrho_1,\varrho_2,\gamma) \geq 0
	\qqand
	g_2(\kappa,\varrho_1,\varrho_2,\gamma) \geq 0.
\end{equation}
\end{proposition}

\begin{proof}
Let $\Zu \in \rectangle$. Then, by Propositions~\ref{proposition:firstIteration}, \ref{proposition:LipschitzStokes} and for $j=2,3$,
\[
\begin{split}
\normInn{\FF_1[\Zu]}_{\frac83} 
&\leq 
\normInn{\FF_1[\Zu]-\FF_1[0]}_{\frac83} +
\normInn{\FF_1[0]}_{\frac83} 
\leq
\paren{1+\frac{\wt{\nu}_1}{\kappa^2}\varrho_1}\alpha_0(\kappa) +
2 \wt{\nu}_2 \varrho_2 \beta_0(\kappa), 
\\
\normInn{\FF_j[\Zu]}_{\frac43} 
&\leq 
\normInn{\FF_j[\Zu]-\FF_j[0]}_{\frac43} +
\normInn{\FF_j[0]}_{\frac43} 
\leq
\frac{\wt{\nu}_3}{\kappa^2}\varrho_1 \alpha_0(\kappa)+
\paren{1+\frac{\wt{\nu}_4 + \wt{\nu}_5}{\kappa^2}\varrho_2}\beta_0(\kappa).
\end{split}
\]
We need to impose that $\FF[\Zu] \in \rectangle$, which leads to the conditions
\[
\normInn{\FF_1[\Zu]}_{\frac83} \leq \varrho_1 \alpha_0(\kappa)
\quad \text{and} \quad
\normInn{\FF_j[\Zu]}_{\frac43} \leq \varrho_2 \beta_0(\kappa).
\]
\end{proof}

\begin{remark}
Let $\varrho_1=40$, $\varrho_2=2$ and $\gamma=0.5$ and consider the functions $g_1$ and $g_2$ considered in Proposition~\ref{proposition:Fwelldefined}. Then,
  
\begin{align*}
g_1(6,40,2,0.5) &\approx -0.0626,
& 
g_2(6,40,2,0.5) &\approx -0.0665,
\\
g_1(7,40,2,0.5) &\approx 0.1613, 
&
g_2(7,40,2,0.5) &\approx 0.1836,
\\
g_1(8,40,2,0.5) &\approx 0.2851, 
&
g_2(8,40,2,0.5) &\approx 0.3226.
\end{align*}
  
%
Notice that we take $\varrho_1 = 20\varrho_2$. This ratio is considered because
\[
\lim_{\kappa \to +\infty} \frac{\beta_0(\kappa)}{\alpha_0(\kappa)} 
= \frac{243}8 \paren{\frac29 + 
\frac{14\sqrt{\pi}\,\Gamma(\frac23)}{81\, \Gamma(\frac76)}}
\approx 
20.3323.
\]
However, other ratios may be considered.
\end{remark}


Now we are ready to prove Theorem \ref{thm:FixedPointQuantitative}.

\begin{proof}[Proof of Theorem \ref{thm:FixedPointQuantitative}]
The first step to prove the theorem is to choose suitable constant so that we can apply Proposition \ref{proposition:Fwelldefined}. 

Indeed, choosing  $\kappa=\kappa^*=6.24$ (see \eqref{def:kappastar}), $\gamma=0.5$, $\varrho_1=20\varrho_2$ and $\varrho_2=1.9$, one has that \eqref{def:g1g2condition} is satisfied. Indeed
\[
g_1(\kappa,\varrho_1,\varrho_2,\gamma) > 0.0371
\qquad \mbox{and}\qquad 
g_2(\kappa,\varrho_1,\varrho_2,\gamma) > 0.0047.
\]
These values are chosen to obtain a small value for $\rho^*(\kappa^*,\gamma)$ (see \eqref{def:rho}). 

Moreover, the constant $L$ in \eqref{def:L} satisfies 
\[
0<L\leq 0.93 < 1
\]
and therefore, by Proposition \ref{proposition:LipschitzStokes}, the operator $\mathcal{F}$ is contractive from $\rectangle$ to itself. Thus, it has a unique fixed point. This completes the proof of Theorem \ref{thm:FixedPointQuantitative}.
\end{proof}

\section{Difference between the solutions of the inner equation}
\label{sec:difference}


To prove Theorem \ref{thm:difference}, the first step is to provide for a good approximation of the solution of the inner equation ``close to infinity''. 
To this end, we define the domains
\[
	\begin{split}
		&\DuInnInf = \claus{ U \in \DuInn \text{ : }
			 \Re U \leq -\eta},
            \quad 
		\DsInnInf = -\DuInnInf,
	\end{split}
\]
where $\DuInn$, $\DuInn$ are the domains introduced in \eqref{def:domainInnner} and $\eta>\kappa$.

We provide an analogue of Theorem \ref{thm:FixedPointQuantitative} in these smaller domains with large $\eta$. This will provide lower values for the constants $b_1$, $b_2$. 

\begin{proposition}\label{prop:fixedpointinfty}
The functions 	$\Zd(U) =(\Wd(U),\Xd(U),\Yd(U))^T$, $\diamond=\unstable,\stable$, introduced in Theorem \ref{theorem:mainAnalytic} are defined in $\mathcal{D}^\diamond_{\kappa^*, \eta^*}$ for 
\[
\kappa^*=6.24,\quad 
\gamma = \frac{1}{2} \quad \text{ and } \quad \eta^*=1000.
\]
In addition, they satisfy that, for $U \in \mathcal{D}^\diamond_{\kappa^*, \eta^*}$,
	\[
		| U^{\frac{8}{3}} \Wd(U)| \leq \tilde b_1, \qquad
		| U^{\frac{4}{3}} \Xd(U) | \leq \tilde b_2, \qquad
		| U^{\frac{4}{3}} \Yd(U) | \leq \tilde b_2,
	\]
where 
\[
\tilde b_1 \leq 0.7,
\qquad
\tilde b_2 \leq 0.71.
\]
\end{proposition}
The proof of this proposition follows exactly the same lines as the proof of Theorem~\ref{thm:FixedPointQuantitative}. Indeed, it is enough to point out that the only difference is that to prove the theorem we strongly used that 
\[
\mathrm{dist}(\DuInn,0)\geq \kappa,
\]
whereas now 
\[
\mathrm{dist}(\DuInnInf,0)\geq \eta.
\]
Taking this fact into account, the proof goes through verbatim replacing $\kappa$ by $\eta$ in the estimates.

To validate the condition (\ref{eq:Delta-Z-at-kappa-star}) from Theorem~\ref{thm:difference} we use the bounds
on $Z^{\mathrm{u}}$ and $Z^{\mathrm{s}}$ from Proposition
\ref{prop:fixedpointinfty} which are valid in the domains $\mathcal{D}%
_{\kappa^{\ast},\eta^{\ast}}^{\mathrm{u}}$ and $\mathcal{D}_{\kappa^{\ast
},\eta^{\ast}}^{\mathrm{s}}$, respectively, and propagate them to the section
$\left\{  \Re U=0\right\}  $ by means of an interval arithmetic
integrator, and establish that the distance between them is non zero. Our tool of choice for this task is the
CAPD\footnote{http://capd.ii.uj.edu.pl} library \cite{MR4283203}.

The CAPD integrator can work with vector fields defined in reals, so our first
step is to rewrite the vector field for the inner equation
(\ref{eq:systemEDOsInner}) in real form. The method for doing so is to separate
the real and imaginary parts of the equations. To achieve this
aim we consider two additional complex variables%
\[
A=\frac{1}{\sqrt{1+\mathcal{J}\left(  U,W,X,Y\right)  }} \qquad\text{and}%
\qquad B=U^{-\frac{1}{3}}.
\]
With these variables we can introduce the following notation%
\[
\mathcal{\tilde{H}}\left(W,X,Y,A,B\right)  
=
W+XY+\mathcal{\tilde{K}}\left(
W,A,B\right) 
\]
and
\begin{equation}
\mathcal{\tilde{K}}\left(W,A,B\right)  =-\frac{3}{4}\frac{1}{B^{2}}%
W^{2}-\frac{1}{3}B^{2}\left(  A-1\right)  .\label{eq:K-tilde}
\end{equation}

\begin{remark}
Notice that, by introducing $A$ and $B$, we have achieved that
\[
\mathcal{H}\left(U,W,X,Y\right)  =\mathcal{\tilde{H}}\left(W,X,Y,A(U,W,X,Y), B(U)\right).
\]
with $\HH$ as given in \eqref{def:hamiltonianInner}. In addition, $\tilde{\HH}$ is polynomial with the only exception of the term involving
$B^{-2}$. This term will not present problems in the
separation of the real parts from the imaginary parts of the vector field in
the coordinates $(U,W,X,Y,A,B)$, since it is easy to separate complex numbers
$z^{-2}$ and $z^{-3}$ into their real and imaginary parts.
\end{remark}


We also write%
  
\begin{align*}
\mathcal{\tilde{J}}\left(W,X,Y,B\right)   &  =\frac{4}{9}W^{2}B^{2}%
-\frac{16}{27}WB^{4}+\frac{16}{81}B^{6}+\frac{4}{9}\left(  X+Y\right)
B^{3}\left(  W-\frac{2}{3}B^{2}\right) \\
&  \quad-\frac{4}{3}i\left(  X-Y\right)  B^{2}-\frac{1}{3}\left(  X^{2}%
+Y^{2}\right)  B^{4}+\frac{10}{9}XYB^{4}.
\end{align*}
  
Note that%
\[
\mathcal{J}\left(  U,W,X,Y\right)  =\mathcal{\tilde{J}}\left(W,X,Y,B(U)\right)
,
\]
where $\JJ$ is given in \eqref{def:hFunction}

To derive the formulae for the vector field with the two additional variables
$A$ and $B$, we first observe that%
\begin{equation} \label{eq:aux-partials}
\begin{split}
 \frac{\partial A}{\partial x}  &=-\frac{1}{2}\frac{1}{\left(  1+\tilde{\JJ} \left( W,X,Y,B\right)  \right)  ^{\frac32}}\frac{\partial\mathcal{\tilde{J}}}{\partial x}\left( W,X,Y,B\right) 
 \\
 &=-\frac12{A^3}\frac{\partial{\tilde{\JJ}}}{\partial
x}\left( W,X,Y,B\right),
\qquad\text{for }x\in\left\{ W,X,Y,B\right\},
\\
\frac{\partial A}{\partial U}  &  =\frac{\partial A}{\partial B}(W,X,Y,B)\frac{\partial
B}{\partial U}(B), 
\qquad \text{where} \qquad
\frac{\partial B}{\partial U} (B)    =-\frac{1}{3}B^{4}.   
\end{split}
\end{equation}
We can now write the ODE (\ref{eq:systemEDOsInner}) in the new variables as
\begin{equation}
\begin{split}
\dot{U}  &  =\phantom{-i}\frac{\partial \HH}{\partial W}=\frac{\partial
\tilde{\HH}}{\partial W}+\frac{\partial\tilde{\HH}}{\partial A}\frac{\partial
A}{\partial W},\nonumber\\
\dot{W}  &  =-\phantom{i}\frac{\partial \HH}{\partial U}=-\left(
\frac{\partial\tilde{\HH}}{\partial B}\frac{\partial B}{\partial U}%
+\frac{\partial\tilde{\HH}}{\partial A}\frac{\partial A}{\partial U}\right)
,\nonumber\\
\dot{X}  &  =\phantom{-}i\frac{\partial \HH}{\partial Y}=\phantom{-}i\left(
\frac{\partial\tilde{\HH}}{\partial Y}+\frac{\partial\tilde{\HH}}{\partial A}%
\frac{\partial A}{\partial Y}\right)  ,\label{eq:ode-AB}\\
\dot{Y}  &  =-i\frac{\partial H}{\partial X}=-i\left(  \frac{\partial
\tilde{\HH}}{\partial X}+\frac{\partial\tilde{\HH}}{\partial A}\frac{\partial
A}{\partial X}\right)  ,\nonumber\\
\dot{A}  &  =\phantom{-i}\frac{\partial A}{\partial U}\dot{U}%
+\frac{\partial A}{\partial W}\dot{W}+\frac{\partial A}{\partial X}\dot{X} +\frac{\partial A}{\partial Y}\dot{Y},\nonumber\\
\dot{B}  &  =\phantom{-i}\frac{\partial B}{\partial U}\dot{U}.\nonumber
\end{split}
\end{equation}

Before we discuss expressing (\ref{eq:ode-AB}) as a vector field on reals, let
us sidetrack to make a useful comment. To do so, let us introduce a function
$S:\mathbb{C}^{6}\rightarrow\mathbb{C}^{6}$ defined as%
\[
S\left(  U,W,X,Y,A,B\right)  :=\left(  -\bar{U},\bar{W},-\bar{X},-\bar{Y}%
,\bar{A},-\bar{B}\right)  .
\]
(In the above, for $z\in\mathbb{C}$ we use $\bar{z}$ for the complex
conjugate.) It turns out that (\ref{eq:ode-AB}) has a time reversing symmetry
with respect to $S$. This will be useful later on, and is expressed in the
following lemma and a resulting corollary.

\begin{lemma}
Let $F:\mathbb{C}^{6}\rightarrow\mathbb{C}^{6}$ stand for the right hand side
of the vector field (\ref{eq:ode-AB}). Then%
\[
S\circ F=-F\circ S. 
\]
\end{lemma}

\begin{proof}
The result follows from (lengthy but elementary) direct validation.
\end{proof}

\begin{corollary}
\label{cor:S-symmetry}The manifolds $Z^{\mathrm{u}}=\left(  W^{\mathrm{u}%
},X^{\mathrm{u}},Y^{\mathrm{u}}\right)  $ and $Z^{\mathrm{s}}=\left(
W^{\mathrm{s}},X^{\mathrm{s}},Y^{\mathrm{s}}\right)  $ are symmetric in the
following sense
\[
W^{\mathrm{s}}\left(  -\bar{U}\right)  =\overline{\Wu\left(
U\right)  },\quad X^{\mathrm{s}}\left(  -\bar{U}\right)  =-\overline
{X^{\mathrm{u}}\left(  U\right)  },\quad Y^{\mathrm{s}}\left(  -\bar
{U}\right)  =-\overline{Y^{\mathrm{u}}\left(  U\right)  }, \quad\text{for }%
U\in\mathcal{D}_{\kappa^{\ast}}^{\mathrm{u}}.
\]

\end{corollary}

Let us also observe that from (\ref{eq:aux-partials}) it follows that
the right hand side
of the vector field~\eqref{eq:ode-AB} do not depend on $U$, which means that we can consider only
the five coordinates $W,X,Y,A,B$ obtaining an ODE in $\mathbb{C}^{5}$;
instead of an ODE in $\mathbb{C}^{6}$.

The right hand side of (\ref{eq:ode-AB}) is `almost' polynomial, with the
exception of the terms involving $B^{-2}$ and $B^{-3}$,
 (these terms come from $\mathcal{\tilde{K}}$ and its derivatives, see (\ref{eq:K-tilde})). Since for a complex
number $z=a+ib$ we have explicit formulae for the real and imaginary parts of
$z^{-2}$ and $z^{-3}$, namely%
\begin{equation}\label{eq:z^2}
\begin{split}
z^{-2}  &  =\left(  a+ib\right)  ^{-2}=\left\vert z\right\vert ^{-4}\left(
a^{2}-b^{2}-i2ab\right)  ,\\
z^{-3}  &  =\left(  a+ib\right)  ^{-3}=\left\vert z\right\vert ^{-6}\left(
a\left(  a^{2}-3b^{2}\right)  +ib\left(  b^{2}-3a^{2}\right)  \right)  ,
\end{split}
\end{equation}
we see that a computation of elementary sums and products of complex
numbers, combined with (\ref{eq:z^2}), leads to the separation
of the real and complex parts on the right hand side of (\ref{eq:ode-AB}).

Such computation is laborious. (Especially for the formula for $\dot{A}$.) We have not performed it by hand, but have used Wolfram Mathematica
\cite{Mathematica} to perform these manipulations. We emphasise that this does
not require any sophisticated computations apart from multiplying complex
numbers and grouping the resulting terms into real and complex parts. We treat
the results returned by Wolfram Mathematica as reliable; in fact more reliable
than if they were performed by us by hand. We enclose a short Wolfram
Mathematica script, together with our code, with which we have performed the symbolic derivation of the separation
of (\ref{eq:ode-AB}) into the real and complex parts\footnote{The code for the computer assisted part of the proof is available on the personal web page of MJC.}. We use the resulting vector field for our CAPD interval arithmetic computations.

Our objective is to compute a bound on $\left\Vert \Delta Z\left( -i\rho\right)
\right\Vert$ for some $\rho$ satisfying that $\rho>\rho
_{0}=7.12$. To do so, we first observe that by the $S$-symmetry of the system
(see Corollary \ref{cor:S-symmetry}) we have%
\begin{equation}
\left\Vert \Delta Z\left(  -i\rho\right)  \right\Vert \geq\left\vert
\Re ( \Delta Y\left(  -i\rho\right)  )\right\vert =2\left\vert
\Re ( Y^{\mathrm{u}}\left(  -i\rho\right)  ) \right\vert
.\label{eq:Z-rho-bound}%
\end{equation}
So, it is enough to show that
\begin{equation}
\left\vert \Re (Y^{\mathrm{u}}\left(  -i\rho\right) ) \right\vert
>0,\label{eq:Y-rho-bound}%
\end{equation}
for some $\rho>\rho_0 = 7.12$.

\begin{figure}[ptb]
\begin{center}
\includegraphics[height=5cm]{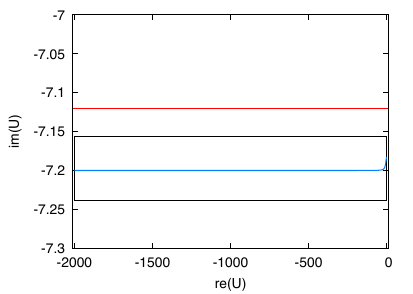}\includegraphics[height=5cm]{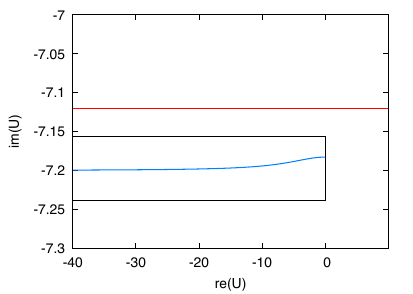}
\end{center}
\caption{The bound on the domain on $U$ within which our trajectory
resides is depicted as the black box. The red line is $\Im U=-\rho_{0}%
$. In blue we have a non-rigorous plot of the trajectory, which is added to
the figure as a point of reference.}\label{fig:crossing-domain}%
\end{figure}

To compute a bound on $Y^{\mathrm{u}}\left(  -i\rho\right)  $ we proceed as
follows. First, we choose an initial point%
\[
U_{0}:=-2000-i \rho_0.
\]
Then, from Proposition \ref{prop:fixedpointinfty} we know that $Z^{\mathrm{u}%
}\left(  U_{0}\right)  $ is inside of the set%
\[
Z^{\mathrm{u}}\left(  U_{0}\right)  \in\mathbf{Z}_{0}:=\left\{  \left(
W,X,Y\right)  :\left\vert W\right\vert \leq\tilde{b}_{1}\left\vert
U_{0}\right\vert ^{-\frac83},\,\left\vert X\right\vert \leq\tilde{b}_{2}\left\vert
U_{0}\right\vert ^{-\frac43},\,\left\vert Y\right\vert \leq\tilde{b}_{2}\left\vert
U_{0}\right\vert ^{-\frac43}\right\}  .
\]
Let us write $\Gamma=\{\Gamma(t)\}$ for the trajectory starting from $(U_0,Z^{\mathrm{u}}(U_0))$. Such trajectory is contained in the unstable manifold.
We have  obtained a bound on $\Gamma$ by
integrating the ODE (\ref{eq:ode-AB}) in interval arithmetic, with the initial condition chosen as the set $\mathbf{Z}_{0}\times\mathbf{A}%
_{0}\times\mathbf{B}_{0}$, where
\[
\mathbf{A}_{0}:=\left\{  \frac{1}{\sqrt{1+\mathcal{J}\left(  U_{0},Z\right)
}}\, \bigg| \, Z\in\mathbf{Z}_{0}\right\}  \qquad\text{and}\qquad\mathbf{B}_{0}:=\left\{
 U_{0}  ^{-1/3}\right\}  .
\]
We make sure that the interval arithmetic bound on $\Gamma$ is always in $\left\{\Im U<\rho_{0}\right\}  $ (see Figure \ref{fig:crossing-domain}) and that it passes through\footnote{In our computer program we have validated that we cross the section $\{\Re U=0\}$ by using the Bolzano type argument, which is visualised in Figure \ref{fig:crossing-CAP}. The CAPD library does have a built in method  for obtaining bounds for a flow reaching a prescribed section, which is transverse to the flow \cite{MR4395996}, but these have failed in the case of our problem. This is why we have obtained the bound for crossing of $\{\Re U=0\}$ without their use. With our present technology the integration to the chosen section is close to the limit of what is achievable for us. This in particular means that the bounds obtained in Proposition \ref{prop:fixedpointinfty} can not easily be extended beyond the chosen surface of section.}
 $\left\{  \Re%
U=0,\Im U\in\lbrack-7.186,-7.18]\right\}  $ (see Figure
\ref{fig:crossing-CAP}, left). We also obtain the following bound (see Figure
\ref{fig:crossing-CAP}, right)
\[
\Re \pi_{Y} ( \Gamma \cap \{\Re U=0\}) 
\in \left[
-0.00075,-0.0005\right],
\]
where $\pi_Y$ denotes the projection into the $Y$ coordinate,
which means that
\[
\Re (Y^{\mathrm{u}}\left(  -i\rho\right) ) \in\left[
-0.00075,-0.0005\right]  ,\quad\text{for }\rho\in\lbrack-7.186,-7.18].
\]
This implies (\ref{eq:Y-rho-bound}) and we thus obtain (\ref{eq:Z-rho-bound}%
). This concludes the proof of Theorem \ref{thm:difference}.

\begin{figure}[ptb]
\begin{center}
\includegraphics[height=5cm]{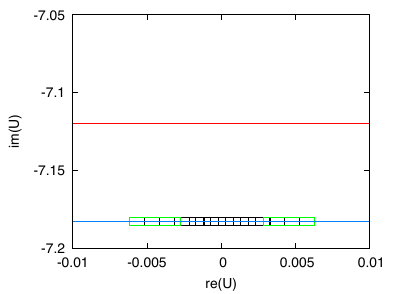}\includegraphics[height=5cm]{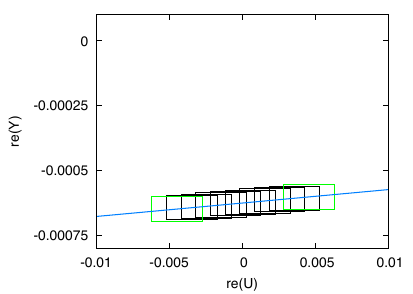}
\end{center}
\caption{A closeup of the crossing of the trajectory through the section
$\{\Re U=0\}$ projected onto $U$ on the left (compare with Figure
\ref{fig:crossing-domain}), and onto coordinates $(\Re U, \Re Y$ on the right. In black we have the interval arithmetic
bounds. In green, we have singled out the bounds on the trajectory for two
disjoint time intervals, to demonstrate that it indeed does cross
$\{\Re U =0\}$. In blue we have a non-rigorous plot of the trajectory,
which is added to the figure as a point of reference.}%
\label{fig:crossing-CAP}%
\end{figure}

The computer assisted computation took a minute on a standard laptop. The vast majority of this time was spent  to integrate in interval arithmetic from $U_0$ to reach the section $\{\Re U=0\}$. (Such integration requires to move along the flow for a time roughly equal to $2000$; equal to the distance between $U_0$ and the section).


\section*{Acknowledgements}

I. Baldom\'a has been supported by the grant PID-2021-
122954NB-100 funded by the Spanish State Research Agency through the programs
MCIN/AEI/10.13039/501100011033 and “ERDF A way of making Europe”.

M. Capi\'nski has been partially supported by the Polish National Science Center (NCN) grants 2019/35/B/ST1/00655 and 2021/41/B/ST1/00407. 

M. Giralt has been supported by the European Union’s Horizon 2020 research and innovation programme under the Marie Sk\l odowska-Curie grant agreement No 101034255.
M. Giralt has also been supported by the research project PRIN 2020XB3EFL ``Hamiltonian and dispersive PDEs".

M. Guardia has been supported by the European Research Council (ERC) under the European Union's Horizon 2020 research and innovation programme (grant agreement No. 757802). M. Guardia is also supported by the Catalan Institution for Research and Advanced Studies via ICREA Academia Prizes 2018 \& 2023. 

This work is supported by the Spanish State Research Agency, through the Severo Ochoa and María de Maeztu Program for Centers and Units of Excellence in R\&D (CEX2020-001084-M).
This work is also part of the grant PID-2021-122954NB-100 funded by MCIN/AEI/10.13039/501100011033 and ``ERDF A way of making Europe''. 
%

\appendix
\section{Explicit expressions for the remainder \texorpdfstring{$\RRR$}{R}}
\label{appendix:formulasRRR}

We devote this appendix to provide formulas for the derivatives of the function $\RRR$ introduced in~\eqref{def:operatorRRRInner}.

We denote $Z=(W,X,Y)$ and assume $U \in \DuInn$ (see \eqref{def:domainInnner}).
The function $\RRR$ is defined as
\begin{equation}\label{eq:operatorRRRStokes}
	\RRR[Z](U) = \paren{\frac{f_1(U,Z)}{1+g(U,Z)},\,
		\frac{\wt{f}_2(U,Z)}{1+g(U,Z)}, \,
		\frac{\wt{f}_3(U,Z)}{1+g(U,Z)} },
\end{equation}
where
\[
	\begin{split}
		\wt{f}_2(U,Z) &= f_2(U,Z) - i{X} g(U,Z), \qquad
		\wt{f}_3(U,Z) = f_3(U,Z) + i{Y} g(U,Z)
	\end{split}
\]
and
\[
	f = \paren{-\partial_U \KK, 
		i \partial_Y \KK, -i\partial_X \KK }^T,
	\qquad
	g = \partial_{W} \KK
\]
and $\KK$ is the Hamiltonian given in~\eqref{def:hamiltonianK} in terms of the function $\JJ$ (see~\eqref{def:hFunction}).

To give formulas for the derivatives of $\RRR$, we first compute the second derivatives of $\JJ$ and $\KK$.

%
%
%

\paragraph{Formulae for $\JJ(U,Z)$.}

The function $\JJ$ given in \eqref{def:hFunction} is defined as
\[
	\begin{split}
		\JJ(U,Z) =& \,  
		\frac{4 W^2}{9 U^{\frac{2}{3}} } 
		-\frac{16 W}{27 U^{\frac{4}{3}}}  
		+\frac{16}{81 U^{2}}
		%
		+\frac{4(X+Y)}{9 U}
		\paren{W -\frac{2}{3 U^{\frac{2}{3}}}} \\[0.5em]
		&- \frac{4i(X-Y)}{3 U^{\frac{2}{3}}}
		-\frac{X^2+Y^2}{3 U^{\frac{4}{3}}}
		+\frac{10 XY}{9 U^{\frac{4}{3}}}. 	
	\end{split}
\]
Then, its first derivatives are given by
\[
\begin{split}
	\partial_U \JJ(U,Z) =& 
	-\frac{8 W^{2}}{27 U^{\frac{5}{3}}} 
	+ \frac{64 W}{81 U^{\frac{7}{3}}}
	-\frac{32}{81 U^{3}}
	-\frac{4 (X+Y) W }{9 U^{2}} \\
	&+\frac{40 (X+Y)}{81 U^{\frac{8}{3}}}
	+\frac{8i (X-Y)}{9 U^{\frac{5}{3}}}
	+\frac{4 (X^{2}+Y^{2})}{9 U^{\frac{7}{3}}}
	-\frac{40 X Y}{27 U^{\frac{7}{3}}}, 
	\\[0.6em]
	\partial_W \JJ(U,Z) =& \,
	\frac{8 W}{9 U^{\frac{2}{3}}}
	-\frac{16}{27 U^{\frac{4}{3}}}
	+\frac{4 (X+Y)}{9 U},
	\\[0.6em]
	\partial_X \JJ(U,Z) =& \,
	\frac{4W}{9U} 
	- \frac{8}{27 U^{\frac53}}
	-\frac{4i}{3 U^{\frac{2}{3}}}
	-\frac{2 X}{3 U^{\frac{4}{3}}}
	+\frac{10 Y}{9 U^{\frac{4}{3}}},
	\\[0.6em]
	\partial_Y \JJ(U,Z) =& \,
	\frac{4W}{9U} 
	- \frac{8}{27 U^{\frac53}}
	+\frac{4i}{3 U^{\frac{2}{3}}}
	-\frac{2 Y}{3 U^{\frac{4}{3}}}
	+\frac{10 X}{9 U^{\frac{4}{3}}}
\end{split}
\]
and the second derivatives are given by
\[
\begin{split}
	\partial_{UW} \JJ(U,Z) =& 
	-\frac{16 W}{27 U^{\frac{5}{3}}}
	+\frac{64}{81 U^{\frac{7}{3}}}
	-\frac{4 (X+Y)}{9 U^{2}},
	\\[0.6em]
	\partial_{UX} \JJ(U,Z) =&
	-\frac{4 W}{9 U^{2}}
	+ \frac{40 }{81 U^{\frac83}}
	+\frac{8 i}{9 U^{\frac{5}{3}}}
	+\frac{8 X}{9 U^{\frac{7}{3}}}-\frac{40 Y}{27 U^{\frac{7}{3}}},
	\\[0.6em]
	\partial_{UY} \JJ(U,Z) =&
	-\frac{4 W}{9 U^{2}}
	+ \frac{40 }{81 U^{\frac83}}
	-\frac{8i}{9 U^{\frac{5}{3}}}
	+\frac{8 Y}{9 U^{\frac{7}{3}}}
	-\frac{40 X}{27 U^{\frac{7}{3}}},
\end{split}
\]
\begin{align*}
	\partial_{W}^2 \JJ(U,Z) &=
	\frac{8}{9 U^{\frac{2}{3}}}, &
	\quad
	\partial_{WX} \JJ(U,Z) &= \frac{4}{9 U}, &
	\quad
	\partial_{WY} \JJ(U,Z) &= \frac{4}{9 U}, 
	\\[0.6em]
	\partial_{X}^2 \JJ(U,Z) &=
	-\frac{2}{3 U^{\frac{4}{3}}}, &
	\quad
	\partial_{XY} \JJ(U,Z) &=
	\frac{10}{9 U^{\frac{4}{3}}}, &
	\partial_{Y}^2 \JJ(U,Z) &=
	-\frac{2}{3 U^{\frac{4}{3}}}. 
\end{align*}

\paragraph{Formulae for $\KK$}

The Hamiltonian $\KK$ introduced in \eqref{def:hamiltonianK} is given by
\[
	\KK(U,Z) = 
	-\frac{3}{4}U^{\frac{2}{3}} W^2 
	- \frac{1}{3 U^{\frac{2}{3}}}
	\paren{\frac{1}{\sqrt{1+\JJ(U,Z)}} - 1 }.
\]
Then, its first derivatives are
\[
\begin{split}
	\partial_U \KK(U,Z) &=
	-\frac{W^2}{2 U^{\frac{1}{3}}}
	+\frac{2}{9 U^{\frac{5}{3}}} 
	\paren{\frac1{\sqrt{1+\JJ}}-1}
	+\frac{1}{6 U^{\frac{2}{3}}} \frac{\partial_U \JJ}
	{(1+\JJ)^{\frac{3}{2}}}, 
	\\
	&=-\frac{W^2}{2 U^{\frac{1}{3}}}
	-\frac{2}{9 U^{\frac{5}{3}}}  \frac{\JJ}{\sqrt{1+\JJ}(1+\sqrt{1+\JJ})}
	+\frac{1}{6 U^{\frac{2}{3}}} \frac{\partial_U \JJ}
	{(1+\JJ)^{\frac{3}{2}}}, 
	\\
	\partial_W \KK(U,Z) &=
	-\frac{3}{2}U^{\frac{2}{3}} W
	+ \frac{1}{6 U^{\frac{2}{3}}}
	\frac{\partial_W \JJ}{(1+\JJ)^{\frac32}},
	\\
	\partial_X \KK(U,Z) &=
	\frac{1}{6 U^{\frac{2}{3}}}
	\frac{\partial_X \JJ}{(1+\JJ)^{\frac32}},
	\\
	\partial_Y \KK(U,Z) &=
	\frac{1}{6 U^{\frac{2}{3}}}
	\frac{\partial_Y \JJ}{(1+\JJ)^{\frac32}}
\end{split}
\]
and its second derivatives are
\[
\begin{split}
	\partial_{UW} \KK(U,Z) &=
	-\frac{W}{U^{\frac{1}{3}}}
	-\frac1{9 U^{\frac{5}{3}}} \frac{\partial_W \JJ}{(1+\JJ)^{\frac32}}
	+\frac{1}{6 U^{\frac{2}{3}}} \frac{\partial_{UW} \JJ}
	{(1+\JJ)^{\frac{3}{2}}}
	-\frac{1}{4 U^{\frac{2}{3}}} \frac{\partial_U \JJ \cdot \partial_W \JJ}
	{(1+\JJ)^{\frac{5}{2}}}, 
	\\
	\partial_{UX} \KK(U,Z) &=
	-\frac1{9 U^{\frac{5}{3}}} \frac{\partial_X \JJ}{(1+\JJ)^{\frac32}}
	+\frac{1}{6 U^{\frac{2}{3}}} \frac{\partial_{UX} \JJ}
	{(1+\JJ)^{\frac{3}{2}}}
	-\frac{1}{4 U^{\frac{2}{3}}} \frac{\partial_U \JJ \cdot \partial_X \JJ}
	{(1+\JJ)^{\frac{5}{2}}},
	\\
	\partial_{UY} \KK(U,Z) &=
	-\frac1{9 U^{\frac{5}{3}}} \frac{\partial_Y \JJ}{(1+\JJ)^{\frac32}}
	+\frac{1}{6 U^{\frac{2}{3}}} \frac{\partial_{UY} \JJ}
	{(1+\JJ)^{\frac{3}{2}}}
	-\frac{1}{4 U^{\frac{2}{3}}} \frac{\partial_U \JJ \cdot \partial_Y \JJ}
	{(1+\JJ)^{\frac{5}{2}}},
	\\
	\partial^2_{W} \KK(U,Z) &=
	-\frac{3}{2}U^{\frac{2}{3}}
	+ \frac{1}{6 U^{\frac{2}{3}}}
	\frac{\partial^2_W \JJ}{(1+\JJ)^{\frac32}}
	- \frac{1}{4 U^{\frac{2}{3}}}
	\frac{(\partial_W \JJ)^2}{(1+\JJ)^{\frac52}},
	\\
	\partial_{WX} \KK(U,Z) &=
	\frac{1}{6 U^{\frac{2}{3}}}
	\frac{\partial_{W X} \JJ}{(1+\JJ)^{\frac32}}
	- \frac{1}{4 U^{\frac{2}{3}}}
	\frac{\partial_W \JJ\cdot \partial_X \JJ}{(1+\JJ)^{\frac52}},
	\\
	\partial_{WY} \KK(U,Z) &=
	\frac{1}{6 U^{\frac{2}{3}}}
	\frac{\partial_{W Y} \JJ}{(1+\JJ)^{\frac32}}
	- \frac{1}{4 U^{\frac{2}{3}}}
	\frac{\partial_W \JJ\cdot \partial_Y \JJ}{(1+\JJ)^{\frac52}},
\end{split}
\]
\[
\begin{split}
	\partial^2_{X} \KK(U,Z) &=
	\frac{1}{6 U^{\frac{2}{3}}}
	\frac{\partial^2_{X} \JJ}{(1+\JJ)^{\frac32}}
	- \frac{1}{4 U^{\frac{2}{3}}}
	\frac{ (\partial_X \JJ)^2}{(1+\JJ)^{\frac52}},
	\\
	\partial_{XY} \KK(U,Z) &=
	\frac{1}{6 U^{\frac{2}{3}}}
	\frac{\partial_{X Y} \JJ}{(1+\JJ)^{\frac32}}
	- \frac{1}{4 U^{\frac{2}{3}}}
	\frac{\partial_X \JJ\cdot \partial_Y \JJ}{(1+\JJ)^{\frac52}},
	\\
	\partial^2_{Y} \KK(U,Z) &=
	\frac{1}{6 U^{\frac{2}{3}}}
	\frac{\partial^2_{Y} \JJ}{(1+\JJ)^{\frac32}}
	- \frac{1}{4 U^{\frac{2}{3}}}
	\frac{ (\partial_Y \JJ)^2}{(1+\JJ)^{\frac52}}.
\end{split}
\]

\paragraph{Formulae for the derivatives of $\RRR$}

By the expression of $\RRR=(\RRR_1,\RRR_2,\RRR_3)$ in~\eqref{eq:operatorRRRStokes}, one obtains
\[
\begin{split}
\partial_W \RRR_1[Z](U) &=
- \frac{\partial_{UW} \KK(1+\partial_W\KK) - \partial_U \KK \cdot \partial^2_W \KK}{(1+\partial_W \KK)^2},
\\
\partial_X \RRR_1[Z](U) &=
- \frac{\partial_{UX} \KK(1+\partial_W\KK) - \partial_U \KK \cdot \partial_{WX} \KK}{(1+\partial_W \KK)^2},
\\
\partial_Y \RRR_1[Z](U) &=
- \frac{\partial_{UY} \KK(1+\partial_W\KK) - \partial_U \KK \cdot \partial_{WY} \KK}{(1+\partial_W \KK)^2}.
\end{split}
\]
Analogously, for $\RRR_2$ and $\RRR_3$
\[
\begin{split}
\partial_W \RRR_2[Z](U) &= 
i\frac{(\partial_{WY} \KK - X\cdot\partial^2_{W}\KK)(1+\partial_W \KK) - (\partial_Y\KK - X\partial_W \KK)\partial^2_{W} \KK }{(1+\partial_W \KK)^2},
\\
\partial_X \RRR_2[Z](U) &= 
i\frac{(\partial_{XY} \KK - \partial_W \KK - X\cdot\partial_{W X}\KK)(1+\partial_W \KK) - (\partial_Y\KK - X\partial_W \KK)\partial_{W X} \KK }{(1+\partial_W \KK)^2},
\\
\partial_Y \RRR_2[Z](U) &= 
i\frac{(\partial^2_{Y} \KK - X\cdot\partial_{W Y}\KK)(1+\partial_W \KK) - (\partial_Y\KK - X\partial_W \KK)\partial_{W Y} \KK }{(1+\partial_W \KK)^2},
\\
	\partial_W \RRR_3[Z](U) &= 
	-i\frac{(\partial_{WX} \KK - Y\cdot\partial^2_{W}\KK)(1+\partial_W \KK) - (\partial_X\KK - Y\partial_W \KK)\partial^2_{W} \KK }{(1+\partial_W \KK)^2},
	\\
	\partial_X \RRR_3[Z](U) &= 
	-i\frac{(\partial^2_{X} \KK - Y\cdot\partial_{W X}\KK)(1+\partial_W \KK) - (\partial_X\KK - Y\partial_W \KK)\partial_{W X} \KK }{(1+\partial_W \KK)^2},
	\\
	\partial_Y \RRR_3[Z](U) &= 
	-i\frac{(\partial_{XY} \KK - \partial_W \KK - Y\cdot\partial_{W Y}\KK )(1+\partial_W \KK) - (\partial_X\KK - Y\partial_W \KK)\partial_{W Y} \KK }{(1+\partial_W \KK)^2}.
\end{split}
\]

\bibliographystyle{alpha}

\addcontentsline{toc}{section}{Bibliography}
\bibliography{biblio}

\end{document}